\documentclass{amsart}
\usepackage{import}

\usepackage[utf8]{inputenc}
\usepackage{amsthm}
\usepackage{amsfonts}
\usepackage{graphicx}
\usepackage{amsmath, amssymb}
\usepackage{hyperref,cleveref}
\usepackage{mathrsfs}
\usepackage{multirow}
\usepackage{centernot}
\usepackage{tikz}
\usepackage{tikz-cd}
\usepackage{color}
\usetikzlibrary{arrows}
\usepackage{pgf}
\usepackage{mathtools}
\usepackage{enumerate}
\usepackage{soul}
\usepackage{comment}
\usepackage[myheadings]{fullpage}

\newtheorem{theorem}{Theorem}[section]

\newtheorem{lemma}[theorem]{Lemma}
\newtheorem{proposition}[theorem]{Proposition}
\newtheorem{corollary}[theorem]{Corollary}
\newtheorem{claim}[theorem]{Claim}

\theoremstyle{definition}
\newtheorem{definition}[theorem]{Definition}

\newtheorem{remark}{Remark}

\newcommand{\R}{\mathbb{R}}
\newcommand{\C}{\mathbb{C}}

\newcommand{\E}{\mathbb{E}}

\newcommand{\HH}{\mathbb{H}}

\newcommand{\PP}{\mathbb{P}}
\newcommand{\Z}{\mathbb{Z}}

\DeclareMathOperator{\tr}{\mathrm{tr}\,}

\newcommand{\e}{\varepsilon}

\newcommand{\1}{\mathbf{1}}

\newcommand{\bi}{\mathbf{i}}

\renewcommand{\vec}[1]{\boldsymbol{#1}}

\newcommand{\fB}{\mathfrak{B}}

\newcommand{\fF}{\mathfrak{F}}

\newcommand{\fa}{\mathfrak{a}}

\newcommand{\cB}{\mathcal{B}}

\newcommand{\cD}{\mathcal{D}}
\newcommand{\cE}{\mathcal{E}}

\newcommand{\cI}{\mathcal{I}}

\newcommand{\cM}{\mathcal{M}}

\newcommand{\cR}{\mathcal{R}}

\newcommand{\cT}{\mathcal{T}}
\newcommand{\cU}{\mathcal{U}}
\newcommand{\cV}{\mathcal{V}}

\newcommand{\cX}{\mathcal{X}}
\newcommand{\cY}{\mathcal{Y}}

\newcommand{\GL}{\mathsf{GL}}
\newcommand{\gl}{\mathfrak{gl}}
\newcommand{\U}{\mathsf{U}}

\DeclareMathOperator{\supp}{supp}

\newcommand{\cl}{\mathrm{cl}}

\makeatletter
\newcommand{\vast}{\bBigg@{4}}
\newcommand{\Vast}{\bBigg@{5}}
\makeatother

\newcommand{\wh}[1]{\widehat{#1}}
\newcommand{\wt}[1]{\widetilde{#1}}

\DeclareMathOperator{\diag}{diag}

\let\Re\relax
\DeclareMathOperator{\Re}{Re}
\let\Im\relax
\DeclareMathOperator{\Im}{Im}

\numberwithin{equation}{section}

\title{Extremal singular values of random matrix products and Brownian motion on $\GL(N,\C)$}
\author{Andrew Ahn}

\begin{document}
\maketitle

\begin{abstract}
We establish universality for the largest singular values of products of random matrices with right unitarily invariant distributions, in a regime where the number of matrix factors and size of the matrices tend to infinity simultaneously. The behavior of the largest log singular values coincides with the large $N$ limit of Dyson Brownian motion with a characteristic drift vector consisting of equally spaced coordinates, which matches the large $N$ limit of the largest log singular values of Brownian motion on $\GL(N,\C)$. Our method utilizes the formalism of multivariate Bessel generating functions, also known as spherical transforms, to obtain and analyze combinatorial expressions for observables of these processes.
\end{abstract}

\section{Introduction}

Suppose $X(1),X(2),\ldots$ is a sequence of $N\times N$ independent random matrices, and let
\[ Y(M) := X(M) \cdots X(1). \]
As a natural model for systems exhibiting progressive scattering, the study of random matrix products has motivations from a variety of contexts including chaotic dynamical systems \cite{Fur1960,CPV1993}, deep neural network \cite{PSG2017,Han2018,HN2020}, and wireless communications \cite{TV2004}. If the matrices $X(m)$ are complex and nonsingular, as in this paper, the discrete time (in $M$) stochastic process $\{Y(M)\}_{M\in \Z_{>0}}$ is a random walk on $\GL(N,\C)$.

In this paper, we establish universality of the largest singular values of $Y(M)$ in the limit as the number of matrix factors $M$ and matrix size $N$ tend to infinity with $M \asymp N$. We focus on random complex matrices $X(m)$ which are \emph{right unitarily invariant}, i.e. the distribution of $X(m) U$ matches the distribution of $X(m)$ for any matrix $U$ in the unitary group $\U(N)$. Under assumptions imposing weak concentration of the support in $\R_{>0}$ and nonvanishing of the average (over $m$) variance of the empirical measures of $X(m)$, we show that the fluctuations of the largest singular values match those of the $N\to\infty$ limit of \emph{Brownian motion $\mathsf{Y}^{(N)}(t)$ on $\GL(N,\C)$}.

The approach used in this work is rooted in tools and ideas from integrable probability (summarized in \Cref{ssec:method}). This departs from previous local universality results for random matrices, such as for Wigner matrices (e.g. \cite{EPRSY10,TV11}), which typically employ non-integrable methods.

\subsection{Brownian motion on $\GL(N,\C)$ and Dyson Brownian motion}

We introduce Brownian motion on $\GL(N,\C)$ and describe the $N\to\infty$ limit of its singular values.

\begin{definition}
\emph{Brownian motion on $\GL(N,\C)$} is a diffusion on $\GL(N,\C)$ with infinitesimal generator given by $\tfrac{1}{2} \Delta$ where $\mathsf{Y}^{(N)}(0)$ is the identity and $\Delta$ is the Laplace-Beltrami operator on $\GL(N,\C)$ with respect to the metric induced by the Hilbert-Schmidt inner product
\[ \langle X, Y \rangle := \tr(X Y^*) \]
on the Lie algebra $\mathfrak{gl}(N,\C)$. Equivalently, $\mathsf{Y}^{(N)}(t)$ is the $\GL(N,\C)$-valued stochastic process $\{\mathsf{Y}^{(N)}(t)\}_{t \ge 0}$ satisfying the Stratonovich equation
\[ d\mathsf{Y}^{(N)}(t) =  \mathsf{Y}^{(N)}(t) \circ dW^{(N)}(t), \quad \quad \mathsf{Y}^{(N)}(0) = 1_N\]
where $1_N \in \GL(N,\C)$ is the identity and
\[ W^{(N)}(t) := \sum_{b \in \beta} W_b(t) b \]
is additive Brownian motion on $\gl(N,\C)$, taking $\beta$ to be any orthonormal basis of $\gl(N,\C)$ with respect to the Hilbert-Schmidt inner product.
\end{definition}

The large $N$ limit of $\mathsf{Y}^{(N)}(t)$, in the sense of *-distribution, is the free multiplicative Brownian motion \cite{Kem16}. Additional aspects of the $N\to\infty$ limit are known, including global fluctuations \cite{CK2014} and the limit of the Brown measure (a candidate for the limit of the eigenvalue empirical distribution). In this paper, we consider the $N\to\infty$ limit of the singular values of $\mathsf{Y}^{(N)}(t)$.

Let $\vec{\xi}^{(N)}(t) = (\xi_1^{(N)}(t) \ge \cdots \ge \xi_N^{(N)}(t))$ denote the logarithm of the squared singular values of $\mathsf{Y}^{(N)}(t)$. It is a remarkable fact (see \cite[Corollary 3.3]{JO06}, \cite{Bia95}, and \cite{IS2016}) that the evolution of $\vec{\xi}^{(N)}(t/4)$ coincides with that Dyson Brownian motion with drift:

\begin{theorem}[{\cite[Corollary 3.3]{JO06}}] \label{thm:brownian_nibm}
The process $\boldsymbol{\xi}^{(N)}(t/4)$ evolves as Brownian motion on $\R^N$ with drift
\[ \left(\frac{N-1}{2},\frac{N-3}{2},\ldots,\frac{-N+3}{2},\frac{-N+1}{2}\right), \]
started at the origin, and
conditioned to remain in the set $\{\vec{x} = (x_1,\ldots,x_N): x_1 \ge \cdots \ge x_N\}$.
\end{theorem}

The original statement from \cite{JO06} was in terms of the singular values of $\mathsf{Y}^{(N)}(t)$. We note that \cite{JO06} also provided analogous descriptions for Brownian motions on symmetric spaces $G/K$ where $G$ is a complex semi-simple non-compact connected Lie group with finite center and $K$ is a maximal compact subgroup. In this framework, the drift vector is the sum of the associated positive roots, where our setting corresponds to Type A.

From the identification with Dyson Brownian motion with drift, the process $\vec{\xi}^{(N)}(t)$ admits determinantal spacetime correlation functions with exact formulas for the correlation kernel. In \cite{Joh04}, the $N\to\infty$ limit of $\xi^{(N)}(t)$ was studied via these kernels, where it was shown to exhibit number variance saturation. In this paper, we prove a stronger form of convergence using machinery from \cite{CH14}.

\begin{theorem} \label{thm:dyson_limit}
There exists a limiting infinite line ensemble (in the sense of \cite{CH14}, see \Cref{def:dyson_limit})
\[ \boldsymbol{\xi}(t) := (\xi_1(t), \xi_2(t),\ldots) = \lim_{N\to\infty} \boldsymbol{\xi}^{(N)}\left( \frac{t}{4} \right) - \frac{Nt}{2} - \log N
\]
where the convergence holds in the following sense. For any $T > 0$ and any positive integer $k$, the random continuous function $\left(\xi_1^{(N)}(\tfrac{t}{4}) - \tfrac{Nt}{2} - \log N,\ldots,\xi_k^{(N)}(\tfrac{t}{4}) - \tfrac{Nt}{2} - \log N \right)$ converges to $(\xi_1(t),\ldots,\xi_k(t))$ as $N\to\infty$ in the weak-* topology of probability measures on $C([\tfrac{1}{T},T])^k$
\end{theorem}

In view of \Cref{thm:dyson_limit}, the limiting process $\vec{\xi}(t)$ may be interpreted as a $\Z_{\ge 1}$-tuple of non-intersecting Brownian motions with drift where the $i$th Brownian motion (from the top) has drift $-i + \tfrac{1}{2}$. The joint Laplace transform and corrlation kernel of $\vec{\xi}(t)$ can be explicitly computed, given in \Cref{thm:line_ensemble}.

\begin{figure}[ht]
    \centering
  \centering
  \includegraphics[width=0.49\linewidth]{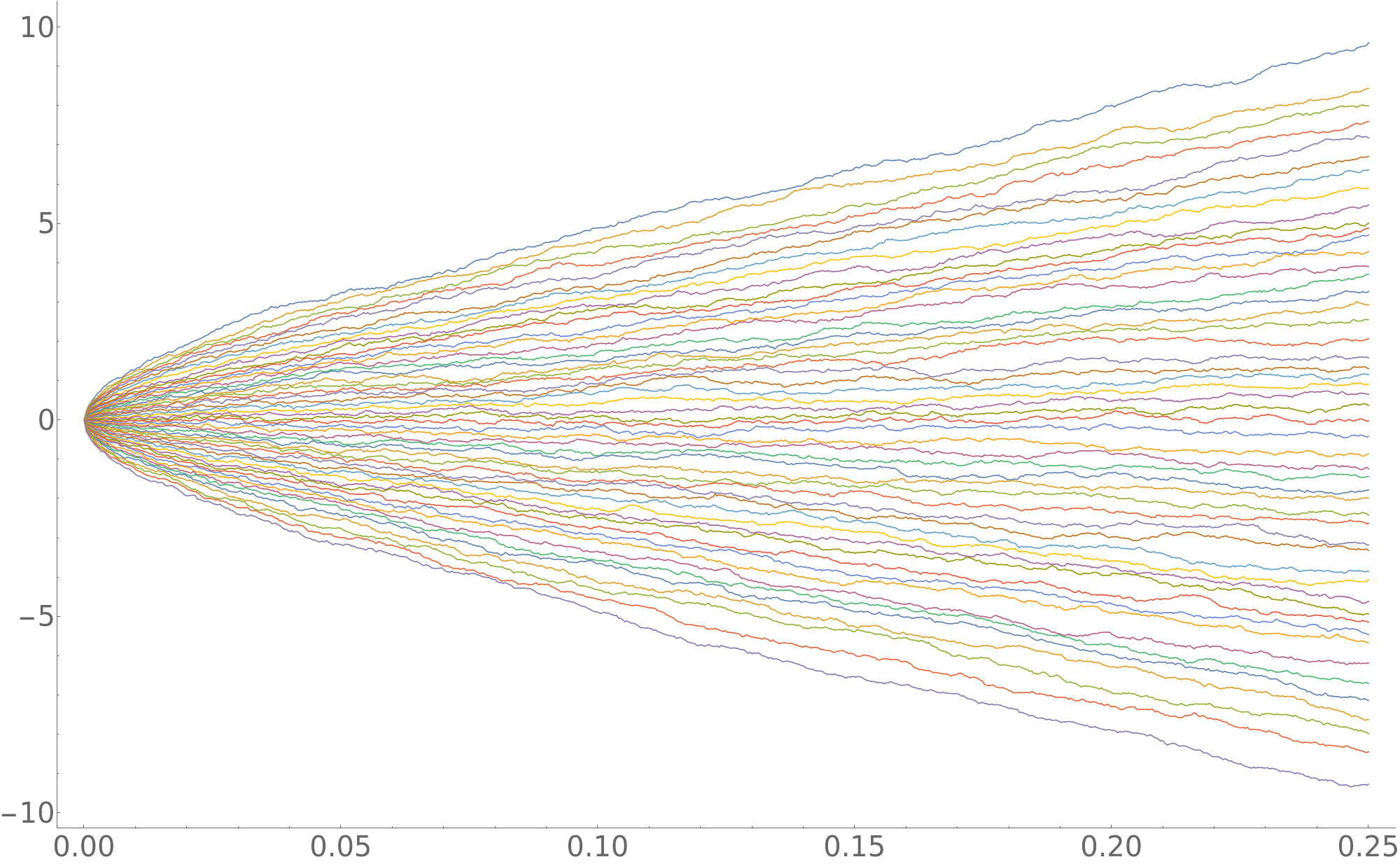}
  \centering
  \includegraphics[width=0.49\linewidth]{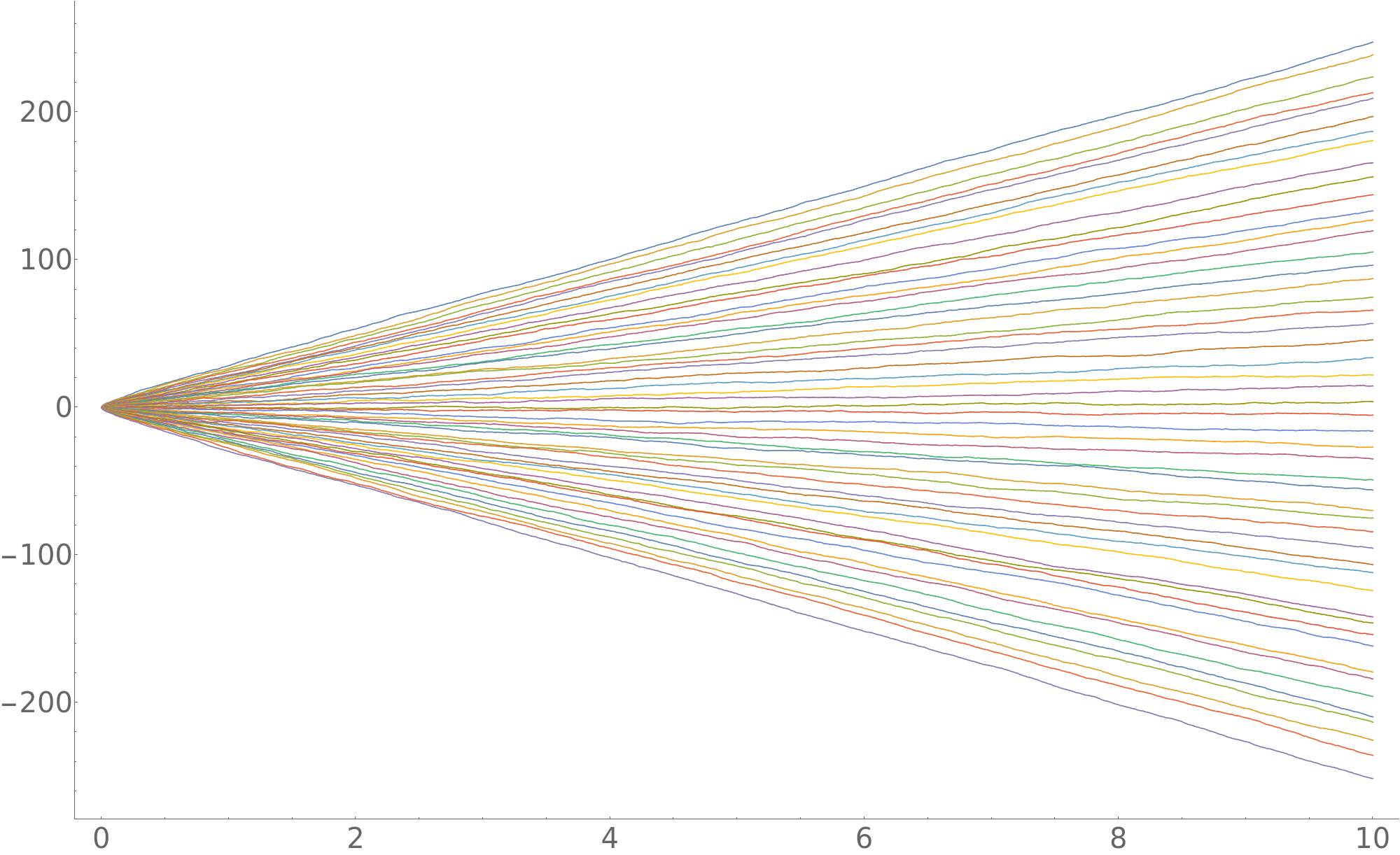}
\caption{Samples of $\vec{\xi}^{(N)}(\tfrac{t}{4})$ where $N = 50$, with $0 \le t \le 0.25$ (left) and $0 \le t \le 10$ (right)}
\label{fig:contour}
\end{figure}

\subsection{Main result}
Our main result is that the largest log squared singular values of random matrix products $Y^{(N)}(M)$ in the limit $N,M\to\infty$ with $N\asymp M$ converges to the infinite line ensemble $\{\boldsymbol{\xi}(t)\}_{t > 0}$ in finite dimensional distribution, under mild assumptions. Given an $N\times N$ matrix $X$, let
\[ |X|^2 := X^*X. \]

\begin{theorem} \label{thm:main}
Suppose $X^{(N)}(1),X^{(N)}(2),\ldots$ are random complex $N\times N$ matrices with right unitarily invariant distributions and denote by
\[ y^{(N)}_1(M) \ge \cdots \ge y^{(N)}_N(M) \]
the squared singular values of $X^{(N)}(M) \cdots X^{(N)}(1)$. Assume that
\begin{enumerate}[(i)]
    \item there exists $C > 0$ such that
    \begin{align} \label{eq:interval_containment}
    1 - \PP\left( \mbox{all squared singular values of $X^{(N)}(m)$ are contained in $[C^{-1},C]$} \right) = o(1/N)
    \end{align}
    uniformly over $m = 1,2,\ldots$, and
    
    \item there exists a continuous $\gamma: \R_{>0} \to \R_{>0}$ such that $\lim_{t \searrow 0} \gamma(t) = 0$ and
    \begin{align} \label{eq:variance_curve}
    \frac{1}{N} \sum_{m=1}^{\lfloor tN \rfloor} \frac{\tr( | X^{(N)}(m)|^4 ) - \tr( |X^{(N)}(m)|^2 )^2}{\tr( |X^{(N)}(m)|^2 )^2} = \gamma(t) + o_{\PP}(1)
    \end{align}
    for each $t > 0$.
    \end{enumerate}
    For each positive integer $h$, the process (in time $t$)
    \[ \log y^{(N)}_j\left( \lfloor tN \rfloor \right) - \sum_{m=1}^{\lfloor tN \rfloor} \log \left( \tr |X^{(N)}(m)|^2 \right) - \log N, \quad \quad j = 1,\ldots,h \]
    converges in finite dimensional distributions to the top $h$ paths $\xi_1(\gamma(t)),\ldots,\xi_h(\gamma(t))$ of $\vec{\xi}(\gamma(t))$.
\end{theorem}

The key assumptions, besides unitary invariance, are given by \eqref{eq:interval_containment} and \eqref{eq:variance_curve}. Condition \eqref{eq:interval_containment} ensures that the largest singular values are sufficiently unlikely to escape to infinity. We expect that this condition can be relaxed to include non-singular matrices, i.e. replace $[C^{-1},C]$ with $[0,C]$. Condition \eqref{eq:variance_curve} ensures that the time parameter of the process converges to a nontrivial deterministic limit, where the individual summand
\[ \frac{1}{N} \frac{\tr( | X^{(N)}(m)|^4 ) - \tr( |X^{(N)}(m)|^2 )^2}{\tr( |X^{(N)}(m)|^2 )^2} \]
is the increment of time that the matrix factor $X^{(N)}(m)$ contributes.

An immediate corollary for the case where the matrix factors are i.i.d. is given below.

\begin{corollary}

Suppose $X^{(N)}(1),X^{(N)}(2),\ldots$ is an i.i.d. sequence of random complex $N\times N$ matrices with right unitarily invariant distributions and denote by
\[ y^{(N)}_1(M) \ge \cdots \ge y^{(N)}_N(M) \]
the squared singular values of $X^{(N)}(M) \cdots X^{(N)}(1)$.
Suppose $X^{(N)}(1),X^{(N)}(2),\ldots$ satisfy \eqref{eq:interval_containment}, and
\begin{align*}
\tr( |X^{(N)}(m)|^2 ) \quad \mbox{and}\quad \frac{\tr( | X^{(N)}(m)|^4 ) - \tr( |X^{(N)}(m)|^2 )^2}{\tr( |X^{(N)}(m)|^2 )^2}
\end{align*}
converge in probability as $N\to\infty$, where the latter has a positive limit $\fa$. For each positive integer $h$, the process (in time $t$)
\[ \log y^{(N)}_j\left( \lfloor tN \rfloor \right) - \sum_{m=1}^{\lfloor tN \rfloor} \log \left( \tr |X^{(N)}(m)|^2 \right) - \log N, \quad \quad j = 1,\ldots,h \]
converges in finite dimensional distributions to the top $h$ paths $\xi_1(\fa t),\ldots,\xi_h(\fa t)$ of $\vec{\xi}(\fa t)$
\end{corollary}

Prior to this work, convergence of the largest log singular values to $\vec{\xi}(t)$ was known for products of Ginibre and truncated unitary matrices \cite{LWW,ABK,Ahn19}. Recall that a \emph{complex Ginibre matrix} is a rectangular matrix of i.i.d. standard complex Gaussian entries, and a \emph{truncated unitary matrix} is a rectangular submatrix of a Haar distributed unitary matrix. The distribution of the former is parametrized by the matrix dimensions and the latter is parametrized by the matrix dimensions and the size of the ambient Haar unitary matrix.

For products of square Ginibre matrices, the convergence of the largest log singular values to $\vec{\xi}(t)$ for fixed time was shown by \cite{ABK,ABK20,LWW}, with generalizations to products of rectangular Ginibre matrices indicated in \cite{LWW}. Extensions to joint time convergence and for products of truncated unitary matrices were established in \cite{Ahn19}. The accessibility of these examples are due to determinantal and related structures available in those cases \cite{AKW2013,AIK2013,KKS16,BGS2018,Ahn19}. In our setting, this structure is not available in general, thus we appeal to alternative methods which we detail later.

To illustrate the relation between \Cref{thm:main} with these previously established results, we briefly review the case for products of Ginibre matrices. This is simpler to state than analogous results for truncated unitary products as it involves fewer parameters (see \cite[Theorem 1.7]{Ahn19} for details). 

Given a sequence of positive integers $\{N_i:= N_i(N)\}_{i \ge 0}$ depending on $N$, where $N_0 = N$ and $N_i \ge N$, let
\[ X^{(N)}(m) = ((G^{(N)}(m))^*(G^{(N)}(m))^{1/2} \]
where $G^{(N)}(m)$ is an $N_m\times N$ complex Ginibre matrix ($m \ge 1$), and consider the associated process $\{Y^{(N)}(M)\}_{M \in \Z_{\ge 0}}$ of matrix products. Under mild conditions on the parameters $\{N_m(N)\}_{m \in \Z_{\ge 0}}$ (see \cite[Theorem 3.4]{LWW})\footnote{The setup from \cite{LWW} was in terms of an equivalent, more elegant setup where $X^{(N)}(m)$ is a \emph{rectangular} $N_m \times N_{m-1}$ complex Ginibre, see \cite[Appendix A]{Ahn19} for details on this equivalence. We stick to products of \emph{square} matrices to remain consistent with the setting of this paper.} which correspond to condition \eqref{eq:variance_curve} in \Cref{thm:main}, the largest log singular values of $Y^{(N)}(M)$ converge to those of $\vec{\xi}(t)$ under an appropriate time change and translation, in the regime $M \asymp N$. \Cref{thm:main} does not imply this result in general because the singular values of $X^{(N)}(m)$ can get arbitrarily close to $0$ if $N_m/N$ approaches $1$, violating \eqref{eq:interval_containment}. If we include the additional hypothesis that
\begin{align} \label{eq:parameter_condition}
\liminf_{N\to\infty} \inf_{i \ge 1} \frac{N_i(N)}{N} > 1,
\end{align}
i.e. the ratios $N_i/N$ remain separated from $1$, then the result for products of Ginibre matrices now follows from \Cref{thm:main}. %With this additional assumption, \eqref{eq:interval_containment} comfortably holds, where the $o(1/N)$ can be replaced by $O(e^{-c N})$ via large deviations, see \cite[Section 2.6.2]{AGZ10}. From this perspective, the $o(1/N)$ is a very mild condition on the support. %Furthermore, in these regimes, the concentration of measure that the translation term
%\[ \sum_{m=1}^{\lfloor tN \rfloor} \log \left( \tr( |X^{(N)}(m)|^2 ) \right) \]
%can be replaced up to $o_{\PP}(1)$ error with a deterministic (depending on $t$) quantity.

Although \Cref{thm:main} requires the additional assumption \eqref{eq:parameter_condition} to reach the full strength previous results on Ginibre products, the methods in this paper can recover these previous results (without \eqref{eq:parameter_condition}) by the integrability of Ginibre and truncated unitary matrices. However, further development is required to deal with general matrices with some singular values approaching $0$.

\subsection{Discussion}

\Cref{thm:main} appears closely related to the functional central limit theorem for $\GL(N,\C)$. Much like Donsker's invariance theorem for random walks on $\R$, a random walk on a connected Lie group $G$ converges to an appropriate diffusion when the increments approach the identity and the number of steps is properly rescaled \cite{SV1973}. Indeed, making right unitary symmetry and mean zero (of log of the increments) assumptions, random walks on $\GL(N,\C)$ converge to $\mathsf{Y}^{(N)}$ with time parametrization dictated by the increments of the original random walk. However, a key distinction which separates \Cref{thm:main} from the $\GL(N,\C)$ functional limit theorem, besides the fact that $N\to\infty$ in the former, is that the increments are not approaching the identity. The connection with $\mathsf{Y}^{(N)}$ is even more striking when comparing to previous results on global limit shapes \cite{New86a,New86b,IN92} and fluctuations \cite{GS} of products of right unitarily invariant random matrices, where the behavior was shown to \emph{non-universal} and independent of the relative growth between $N$ and $M$. %To summarize, even in the regime where increments do not approach the identity, if the matrix size $N$ tends to infinity we approach $\mathsf{Y}^{(N)}(t)$ local statistics at the edge of the spectrum, a phenomenon which seems absent at the level of global statistics.

This paper focuses on the regime $N,M \to \infty$ where $N \asymp M$. However, for products of Ginibre matrices, the regimes $N \gg M$ and $N \ll M$ are also known \cite{LWW,ABK,ABK20}. For $N \ll M$, the so-called picket fence statistics appear, where the $i$th largest log singular value concentrates near $-i + \tfrac{1}{2}$ after suitable rescaling. For $N \gg M$, GUE statistics appear. Thus the process $\vec{\xi}(t)$ may also be viewed as an interpolating process between these two regimes. The appearance of the picket fence statistics is directly related to separation of the curves in $\vec{\xi}(t)$ as $t \to \infty$ according the drift sequence $-\tfrac{1}{2}, -\tfrac{3}{2}, -\tfrac{5}{2},\ldots$. The appearance of GUE statistics corresponds to forgetting the drift as $t \to 0$.

In view of this description for Ginibre matrices, \Cref{thm:main} suggests (though does not directly imply) that for $N \ll M$, the log squared singular values of the random walk $Y^{(N)}(M)$ converge to the picket fence. This can be recast into a statement about universality of Lyapunov exponents, by interpreting the limit $N \ll M$ as taking limits $M\to\infty$ and $N\to\infty$ in that order. More precisely, Oseledets' multiplicative ergodic theorem \cite{Ose1968} asserts the existence of Lyapunov exponents
\[ \lambda_i^{(N)} = \lim_{M\to\infty} \frac{1}{2M} \log y_i^{(N)}(M), \quad \quad 1 \le i \le N \]
for fixed $N$. While these Lyapunov exponents are not universal, we expect that under general assumptions as $N\to\infty$ the largest Lyapunov exponents should converge to picket fence statistics upon properly rescaling and translating. Results of this type were established for products of truncated unitary and Ginibre matrices in \cite{AVP21}.

Likewise, \Cref{thm:main} also suggests that for $N \gg M$, the largest singular values of $Y^{(N)}(M)$ should converge to GUE statistics given by the Airy point process. In the extreme case where the number of matrix factors $M$ is fixed and $N\to\infty$, the global limit shape can be understood in terms of free probability \cite{V87}. However, local statistics are far less understood, though significant progress \cite{Ji21,DJ20} has been made in the form of regularity and local law results at the edge. The belief is that the Airy point process should appear as long as the empirical distribution of the matrix factors are suitably regular, and that the regularity can be relaxed as the number of matrix factors $M$ increases. A similar phenomenon was confirmed for local statistics of sums of random Hermitian matrices \cite{Ahn20}.

\subsection{Further directions}

A natural question is whether the large $N$ limit of the statistics of $\mathsf{Y}^{(N)}(t)$ appear for products of random matrices beyond the edge. Returning to the case of products of complex Ginibre matrices, it is known \cite{ABK20} that the bulk statistics in the regime $M \asymp N$ continue to the match that of $\mathsf{Y}^{(N)}(t)$, suggesting that this universality holds in the bulk. It is worth noting that for products of square Ginibre matrices, there is a hard edge which is absent for $\mathsf{Y}^{(N)}(t)$. %Therefore, for random matrix products with a hard edge, $\mathsf{Y}^{(N)}(t)$ universality can only persist in the bulk or soft edge. %Nonetheless, there is still a link with the global limit shape and fluctuations of $\mathsf{Y}^{(N)}(t)$ where $t$ is not rescaled as $N\to\infty$ \textcolor{red}{comment or computation?}.

Let us remark on several related models which do not exhibit right unitary symmetry. For products of complex Wishart matrices (matrices with centered, variance $1/n$, i.i.d. random variables not necessarily Gaussian), we conjecture that the large $N$ limit of $\mathsf{Y}^{(N)}(t)$ statistics continue to appear at the edge and bulk as well. In another direction, one can consider other symmetry classes, such as products of real matrices with right orthogonal invariance or of quaternionic matrices with right symplectic invariance. Natural examples are products of real Ginibre matrices and truncated Haar orthogonal matrices for the former, and products of quaternionic Ginibre matrices and truncated Haar symplectic matrices for the latter. These models, and one-parameter deformations of these models in the Dyson index $\beta$ in the spirit of $\beta$-ensembles, were considered in \cite{Ahn19} where tightness results were obtained at the edge in the regime $M \asymp N$. While convergence results beyond $\beta = 2$ are unavailable, we expect $N\to\infty$ limits of $\beta$-deformations of $\mathsf{Y}^{(N)}(t)$ to appear, where for $\beta = 1$ and $4$ this should be the corresponding diffusion on $\GL(N,\R)$ and $\GL(N,\HH)$. respectively (where $\HH$ is the skew field of real quaternions).

Although convergence results for products of right orthogonally invariant real matrices are not available, there is a similar model which may be accessible via existing methods. The eigenvalues of $X_M^T \cdots X_1^T A X_1 \cdots X_M$ where $X_1,\ldots,X_M$ are real Ginibre matrices and $A$ is some real antisymmetric matrix are determinantal \cite{KFI2019,FILZ2019}, structure which arises from the orthogonal Harish-Chandra-Itzykson-Zuber integral. The behavior of the eigenvalues in the regime $M \asymp N$ have not been studied, though may be accessible by analysis of correlation kernels. It would be interesting to see what behavior arises in this regime.

There is also recent progress on singular numbers of products of $p$-adic matrices \cite{Van20} which exhibit some parallels to our setting. In particular, there is a universality phenomenon where objects corresponding to Lyapunov exponents converge in the large $N$ limit to a geometric progression. It is possible that there may be analogues of our result in this setting.

\subsection{Method} \label{ssec:method}
The methods in this paper rely on an analogue of the Mellin transform for distributions of right unitarily invariant matrices. Given a random $X$ in $\GL(N,\C)$, define
\[ \Phi_X(\vec{z}) = \E \left[ \int_{\cU(N)} |U X^*X U^{-1}|^{\vec{z}} \, dU \right], \quad \quad \vec{z} \in \C^N \]
abusing notation by using the random $X$ as a subscript of $\Phi$, where
\[ |Y|^{\vec{z}} := (\det Y)^{z_N} \prod_{j=1}^{N-1} \left( \det Y_{j\times j} \right)^{z_j - z_{j+1} - 1} \]
is a generalization of the power function for positive definition matrices; $Y_{j\times j}$ is the $j\times j$ top left corner submatrix of $Y$ which itself is positive definite by Sylvester's criterion. The expectation satisfies the factorization property
\[ \Phi_{XY}(\vec{z}) = \Phi_X(\vec{z}) \Phi_Y(\vec{z}) \]
where $X,Y$ are independent random right unitarily invariant complex $N\times N$ matrices. These ideas, and their additive analogue, have been used to compute exact density formulas for a variety of matrix ensembles as in \cite{KR2019,KMZ2020,FKK2021,ZKF2021} under a common framework.

The integral within the expectation of $\Phi_X(\vec{z})$ is known as the Gelfand-Na\u{\i}mark integral. It can be explicitly evaluated \cite{GN50}
\[ \int_{\cU(N)} |U\diag(x_1,\ldots,x_N) U^{-1}| \, dU = \Delta(\rho_N) \frac{\det x_j^{z_i}}{\Delta(\vec{z}) \Delta(\vec{x})}, \quad \quad \Delta(u_1,\ldots,u_N) = \prod_{1 \le i < j \le N} (u_i - u_j) \]
and may be viewed as a multiplicative analogue of the Harish-Chandra-Itzykson-Zuber integral \cite{HC57,IZ80}. The right hand side is a (normalized) multivariate Bessel function, a continuous analogue of a Schur function. This connection with symmetric functions yields a collection of tools for the study of right unitarily invariant random matrix products. In particular, we can act on the spherical transforms by certain operators diagonalized by the multivariate Bessel functions to obtain observables for singular values of matrix products. Similar ideas were used to study other random processes with connections to symmetric functions, including polymers \cite{BC14}, measures arising from representation theory \cite{BuG15,BuG18,BuG19}, the $\beta$-Jacobi corners process from $\beta$-ensembles \cite{BG15,GZ18}, and many more. This idea was used by \cite{GS} to study global fluctuations of random matrix products, where the spherical transform was referred to as the multivariate Bessel generating function, as we will in the body of this paper due to methodological connections with their work.

Our method relies on the extraction of joint observables for singular values of $Y^{(N)}(m)$ via the appropriate family of operators. The observables and corresponding operators used in \cite{GS} are not amenable for the analysis of edge statistics in our setting. Thus, we consider a family of operators suitable for our regime corresponding to observables which give the joint Laplace transform of the log singular values of $Y^{(N)}(M)$ (over varying $M$). This leads to considerable differences from the analyses of \cite{GS}. Our observables allow us to access the edge in a similar manner as the method of high moments in \cite{Sos99} probed the edge for Wigner matrices. In short, joint Laplace transforms of the singular values are dominated by the largest singular values in our limit. Using exact formulas, these observables have expressions in terms of large combinatorial sums which can be asymptotically identified with expressions which correspond to the large $N$ limit of observables of $\mathsf{Y}^{(N)}(t)$.

\subsection{Organization}
The remainder of this paper is organized as follows. In \Cref{sec:operators} we introduce the main tools to access the observables of singular values of random matrix products: the multivariate Bessel generating function and associated operators. \Cref{sec:line_ensemble} introduces the formalism of line ensembles and is devoted to proving \Cref{thm:dyson_limit}, along with auxiliary results for later parts of the paper. \Cref{sec:S-transform} introduces the $S$-transform and the $\psi$ function of a measure on $\R_{>0}$, and similarly proves auxiliary results for later parts of the paper. \Cref{sec:asymptotics_bessel} obtains asymptotics of multivariate Bessel functions, bootstrapping off of a result of \cite{GS}, a key input for the asymptotics of the joint Laplace transforms that we want to compute. Finally, \Cref{sec:main_proof} proves the main result \Cref{thm:main}, containing the core asymptotic analysis of this paper.

\section*{Acknowledgements}

I thank Roger Van Peski for many discussions which immensely helped clarify my understanding of this topic, and Yi Sun for introducing me to the literature and early discussions which motivated me to start this work. I would also like to thank Alexei Borodin and Vadim Gorin for helpful comments and suggestions.

%The model $Y^{(N)}(M)$ was introduced as a random walk on $\GL(N,\C)$. However, the right unitary invariance admits an interpretation of $Y^{(N)}(M)$ as a random walk on the symmetric space $\GL(N,\C)/\U(N)$. As described in \cite{Kly00}, the triple of symmetric spaces $\U(N)$, $\mathrm{Herm}(N,\C)$ (the space of complex $N\times N$ Hermitian matrices), $\GL(N,\C)/\U(N)$ It is worth noting the connection between Dyson Brownian motion and symmetric spaces \cite{Kly00}

%%%%%%%%%%%%%%%%%%%%%%%%%%%%%%%%%

\section{Joint Laplace Transforms and Multivariate Bessel Functions} \label{sec:operators}

In this section, we introduce the multivariate Bessel generating function, also known as the spherical transform, which are expectations of random matrices over the multivariate Bessel function, see e.g. \cite{GM18} and references therein. The multivariate Bessel generating functions cohere well with matrix products, a fact which furnishes us with expressions for observables of squared singular values of random matrix products. 

\begin{definition}
The \emph{multivariate Bessel function} indexed by $\vec{a} = (a_1 \ge \cdots \ge a_N) \in \R^N$ is the function
\[ \cB_{\vec{a}}(\vec{z}) = \frac{\det\left[ e^{z_i a_j} \right]_{i,j=1}^N}{\Delta(\vec{z})} \]
which is holomorphic for $\vec{z} \in \C^N$, where
\[ \Delta(\vec{z}) = \prod_{1 \le i < j \le N} (z_i - z_j). \]
\end{definition}

\begin{definition}
Let $X$ be a random matrix in $\GL(N,\C)$ with right unitarily invariant distribution, Denote by $\vec{x} = (x_1,\ldots,x_N)$ the squared singular values of $X$. The \emph{multivariate Bessel generating function} of $X$ is defined by
\[ \Phi_X(z_1,\ldots,z_N) = \E\left[ \frac{\cB_{\log \vec{x}}(z_1,\ldots,z_N)}{\cB_{\log \vec{x}}(\rho_N)} \right], \]
where $\rho_N := (N-1,N-2,\ldots,0)$, given that this expectation exists in a neighborhood of $(N-1,N-2,\ldots,0)$.
\end{definition}

The normalized multivariate Bessel function within the expectation is the zonal spherical function for the Gelfand pair $\GL(N,\C), \U(N)$ \cite[Chapter VII]{Mac}. If $X$ is a scalar (i.e. $N = 1$), then the multivariate Bessel generating function reduces to
\[ \Phi_X(z) = \E\left[|X|^{2z}\right] \]
which is the Mellin transform for the distribution of $|X|^2$. Thus, for general $N$ the function $\Phi_X$ is an extension of the Mellin transform for positive definite matrices $X^*X$. Moreover, we can define a generalized power function on $N\times N$ positive definite matrices $Y$:
\[ |Y|^{\vec{z}} := (\det Y)^{z_N} \prod_{j=1}^{N-1} \left( \det Y_{j\times j} \right)^{z_j - z_{j+1} - 1} \]
where $Y_{j\times j}$ is the $j\times j$ top left corner submatrix of $Y$. Indeed, by Sylvester's criterion, the determinants are positive, thus the complex exponentials are well-defined. Then
\[ \frac{\cB_{\log \vec{x}}(\vec{z})}{\cB_{\log \vec{x}}(\rho_N)} = \Delta(\rho_N) \frac{\det\left[ x_j^{z_i} \right]_{i,j=1}^N}{\Delta(\vec{z}) \Delta(\vec{x})} = \int_{\U(N)} \left| U \diag(\vec{x}) U^{-1} \right|^{\vec{z}} \, dU \]
where the integral is over the normalized (with volume $1$) Haar measure on $\U(N)$. The latter integral is known as the Gelfand-Na\u{\i}mark integral \cite{GN50}. From this perspective, the multivariate Bessel functions are unitarily invariant (under the conjugation action) generalized power functions on positive definite matrices.

Just as products of independent random variables factor under the Mellin transform, the multivariate Bessel generating functions satisfy the following factorization property:

\begin{lemma} \label{thm:spherical_convolution}
If $X$ and $Y$ are independent $N\times N$ random matrices with right unitarily invariant distributions, then
\[ \Phi_{XY} = \Phi_X \cdot \Phi_Y. \]
\end{lemma}

\begin{proof}
If $X$ and $Y$ have deterministic squared singular values $\vec{x} \in \R_{>0}^N$ and $\vec{y} \in \R_{>0}^N$ respectively, then this follows from the identity
\[ \frac{\cB_{\log \vec{x}}(\vec{z})}{\cB_{\log \vec{x}}(\rho_N)} \frac{\cB_{\log \vec{y}}(\vec{z})}{\cB_{\log \vec{x}}(\rho_N)} = \E\left[ \frac{\cB_{\log \vec{w}}(\vec{z})}{\cB_{\log \vec{w}}(\rho_N)} \right] \]
for zonal spherical functions, where $\vec{w}$ are the squared singular values of $XY$, see \cite[Chapter VII]{Mac}. The general case follows from taking mixtures of the aforementioned case.
\end{proof}

We can iterate \Cref{thm:spherical_convolution}. Let $X(1),X(2),\ldots$ be independent $N\times N$ right-unitarily invariant complex random matrices. Given that the multivariate Bessel generating functions $\Phi_{X(m)}$ ($m \ge 1$) exist, the product $Y(M) := X(M) \cdots X(1)$ has multivariate Bessel generating function
\[ \prod_{m=1}^M \Phi_{X(m)}(z_1,\ldots,z_N). \]

Given $c \in \C$, define
\[ \cD_c := \cD_c^{(N)} := \sum_{i=1}^N \left(\prod_{j \ne i} \frac{c + z_i - z_j}{z_i - z_j}\right) \cT_{c,z_i} \]
where $\cT_{c,z_i}f(z_1,\ldots,z_N) = f(z_1,\ldots,z_i + c,\ldots,z_N)$. Then
\begin{align} \label{eq:eigenrelation}
\cD_c \cB_{\log \vec{x}}(z_1,\ldots,z_N) = \left( \sum_{i=1}^N x_i^c \right) \cB_{\log \vec{x}}(z_1,\ldots,z_N).
\end{align}

\begin{proposition} \label{thm:observable_operators}
Let $X(1),X(2),\ldots$ be independent nonsingular $N\times N$ random matrices with right unitarily invariant distributions and multivariate Bessel generating functions $\varphi_1,\varphi_2,\ldots$ respectively. Assume that the multivariate Bessel generating functions are analytic on $\C^N$. Fix real numbers $c_1,\ldots,c_k > 0$ and integers $M_1 \ge \cdots \ge M_k > M_{k+1} = 0$. Suppose $\vec{y}(M) \in \R_{>0}^N$ is the vector of squared singular values of $Y(M) = X(M) \cdots X(1)$. Then
\[ \E\left[ \prod_{i=1}^k \sum_{j=1}^N y_j(M_i)^{c_i} \right] = \left. \cD_{c_1} \prod_{m_1=M_2+1}^{M_1} \varphi_{m_1}(z_1\ldots,z_N) \cdots \cD_{c_k} \prod_{m_k=M_{k+1}+1}^{M_k} \varphi_{m_k}(z_1,\ldots,z_N) \right|_{\vec{z} = \rho_N} \]
where the operators $\cD_c$ act on everything to the right of it.
\end{proposition}

With \Cref{thm:brownian_nibm}, we can compute the multivariate Bessel generating function for Brownian motion on $\GL(N,\C)$ using the following general result:

\begin{proposition} \label{thm:brownian_spherical_function}
Suppose that $\boldsymbol{\eta}(t)$ is the vector of $N$-intersecting Brownian motions with drift $\vec{\mu} = (\mu_1 \ge \cdots \ge \mu_N)$ and $\boldsymbol{\eta}(0) = \vec{0}$. If
\[ \boldsymbol{\eta}(t) + a := (x_1(t) + a,\ldots,x_N(t) + a) \]
for some $a \in \R$, then
\[ \E\left[ \frac{\cB_{\boldsymbol{\eta}(t) + a}(\vec{z})}{\cB_{\boldsymbol{\eta}(t) + a}(\vec{\mu})} \right] = \prod_{i=1}^N \frac{e^{\frac{t}{2}(z_i + \frac{a}{t})^2}}{e^{\frac{t}{2}(\mu_i + \frac{a}{t})^2}}. \]
\end{proposition}

In particular, \Cref{thm:observable_operators,thm:brownian_spherical_function} imply

\begin{corollary} \label{thm:brownian_observable}
Given $c_1,\ldots,c_k > 0$ and $t_1 > \cdots > t_k > t_{k+1} = 0$, we have
\begin{align} \label{eq:brownian_observable}
\E \left[ \prod_{i=1}^k \sum_{j=1}^N e^{c_i\left(\xi_j^{(N)}(\frac{t_i}{4}) - \frac{Nt_i}{2}\right)} \right] = \left. \cD_{c_1} \prod_{i=1}^N \frac{e^{\frac{1}{2}(t_1 - t_2)(z_i - N + \frac{1}{2})^2}}{e^{\frac{1}{2}(t_1 - t_2)(-i + \frac{1}{2})^2}} \cdots \cD_{c_k} \prod_{i=1}^N \frac{e^{\frac{1}{2}(t_k - t_{k+1})(z_i - N + \frac{1}{2})^2}}{e^{\frac{1}{2}(t_k - t_{k+1})(-i + \frac{1}{2})^2}} \right|_{\vec{z} = \rho_N}
\end{align}
where the operators $\cD_c$ act on everything to the right of it. The right hand side can be expressed as a multiple contour intergral:
\begin{align} \label{eq:brownian_contour}
\begin{split}
& \E\left[ \prod_{i=1}^k \sum_{j=1}^N e^{c_i\left( \xi_j(\frac{t_i}{4}) - \frac{Nt_i}{2} \right)} \right] \\
& \quad \quad = \frac{1}{(2\pi\bi)^k} \oint \cdots \oint \prod_{1 \le i < j \le k} \frac{(z_i - z_j)(z_i + c_i - z_j - c_j)}{(z_i - z_j - c_j)(z_i + c_i - z_j)} \prod_{i=1}^k \frac{e^{\frac{t_i}{2}(z_i + c_i - \frac{1}{2})^2}}{e^{\frac{t_i}{2}(z_i - \frac{1}{2})^2}} \frac{\Gamma(z_i + c_i + N) \Gamma(z_i)}{\Gamma(z_i + c_i) \Gamma(z_i + N)} \frac{dz_i}{c_i}
\end{split}
\end{align}
where the $z_i$ contour is positively oriented around $0,-1,\ldots,-N+1$ for $1 \le i \le k$ and the $z_j$ contour contains $z_i + c_i$ and $z_i - c_j$ for each $1 \le i < j \le k$.
\end{corollary}

\begin{remark}
The contour integral formula \eqref{eq:brownian_contour} can be viewed as a special case of \cite[Propositions 2.8 and 2.9]{Ahn20} and is closely related to the formula \cite[Theorem B.2]{Ahn19} for observables of Schur processes. These ideas go further back to the work of \cite{BC14} where Macdonald processes, generalizations of the Schur processes, were introduced to study directed polymers. In this work, a family of contour integral formulas for observables of Macdonald processes were used to access these polymer models. 
\end{remark}

We now provide the proofs of these results.

\begin{proof}[Proof of \Cref{thm:observable_operators}]
We show that
\begin{align} \label{eq:induction_observable}
\begin{split}
& \E\left[ \left( \prod_{i=1}^k \sum_{j=1}^N y_j(M_i)^{c_i} \right) \frac{\cB_{\log \vec{y}(M_1)}(z_1,\ldots,z_N)}{\cB_{\log \vec{y}(M_1)}(\rho_N)} \right] \\
& \quad \quad = \cD_{c_1} \prod_{m_1=M_2+1}^{M_1} \varphi_{m_1}(z_1,\ldots,z_N) \cdots \cD_{c_k} \prod_{m_k=M_{k+1}+1}^{M_k} \varphi_{m_k}(z_1,\ldots,z_N)
\end{split}
\end{align}
by induction on $k$. The result follows from evaluating the expression above at $\vec{z} = \rho_N$. Indeed, \eqref{eq:eigenrelation} and \Cref{thm:spherical_convolution} imply that
\[ \cD_c\left[ \left( \sum_{j=1}^N y_j(M)^c \right) \frac{\cB_{\log \vec{y}(M)}(z_1,\ldots,z_N)}{\cB_{\log \vec{y}(M)}(\rho_N)} \right] = \cD_c \prod_{m=1}^M \varphi_m(z_1,\ldots,z_N) \]
which is the $k = 1$ base step in the induction. Next, suppose we know that
\begin{align} \label{eq:induction_hypothesis_observable}
\begin{split}
& \E\left[ \left( \prod_{i=2}^k \sum_{j=1}^N y_j(M_i)^{c_i} \right) \frac{\cB_{\log \vec{y}(M_2)}(z_1,\ldots,z_N)}{\cB_{\log \vec{y}(M_2)}(\rho_N)} \right] \\
& \quad \quad = \cD_{c_2} \prod_{m_2=M_3+1}^{M_2} \varphi_{m_2}(z_1,\ldots,z_N) \cdots \cD_{c_k} \prod_{m_k=M_{k+1}+1}^{M_k} \varphi_{m_k}(z_1,\ldots,z_N)
\end{split}
\end{align}
which is equivalent to assuming the induction hypothesis for $k-1$. Multiply both sides by
\[ \prod_{m_1 = M_2 + 1}^{M_1} \varphi_{m_1}(z_1,\ldots,z_N) \]
and apply $\cD_{c_1}$. Then the right hand side of \eqref{eq:induction_hypothesis_observable} becomes the right hand side of \eqref{eq:induction_observable}. The left hand side of \eqref{eq:induction_hypothesis_observable} becomes
\begin{align} \label{eq:induction_step_observable}
\begin{split}
& \cD_{c_1} \prod_{m_1 = M_2 + 1}^{M_1} \varphi_{m_1}(z_1,\ldots,z_N) \cdot \E\left[ \left( \prod_{i=2}^k \sum_{j=1}^N y_j(M_i)^{c_i} \right) \frac{\cB_{\log \vec{y}(M_2)}(z_1,\ldots,z_N)}{\cB_{\log \vec{y}(M_2)}(\rho_N)} \right] \\
& \quad \quad = \E\left[ \left( \prod_{i=2}^k \sum_{j=1}^N y_j(M_i)^{c_i} \right) \cD_{c_1} \frac{\cB_{\log \vec{y}(M_2)}(z_1,\ldots,z_N)}{\cB_{\log \vec{y}(M_2)}(\rho_N)} \prod_{m_1 = M_2 + 1}^{M_1} \varphi_{m_1}(z_1,\ldots,z_N) \right]
\end{split}
\end{align}
which we want to match with the left hand side of \eqref{eq:induction_observable}. Observe that
\[ \frac{\cB_{\log \vec{y}(M_2)}(z_1,\ldots,z_N)}{\cB_{\log \vec{y}(M_2)}(\rho_N)} \prod_{m_1 = M_2 + 1}^{M_1} \varphi_{m_1}(z_1,\ldots,z_N) = \E\left[ \left. \frac{\cB_{\log \vec{y}(M_1)}(z_1,\ldots,z_N)}{\cB_{\log \vec{y}(M_1)}(\rho_N)} \right| \vec{y}(M_2) \right] \]
by \Cref{thm:spherical_convolution}. In words, the left hand side is the multivariate Bessel generating function for the matrix product $X(M_1) X(M_1 - 1)\cdots X(M_2 + 1) Y(M_2) = Y(M_1)$ where we condition $Y(M_2)$ to have squared singular values given by $\vec{y}(M_2)$. Using the identity above, \eqref{eq:induction_step_observable} becomes
\begin{align*}
& \E\left[ \left( \prod_{i=2}^k \sum_{j=1}^N y_j(M_i)^{c_i} \right) \cD_{c_1} \E\left[ \left. \frac{\cB_{\log \vec{y}(M_1)}(z_1,\ldots,z_N)}{\cB_{\log \vec{y}(M_1)}(\rho_N)} \right| \vec{y}(M_2) \right] \right] \\
& \quad \quad = \E\left[ \left( \prod_{i=2}^k \sum_{j=1}^N y_j(M_i)^{c_i} \right) \E\left[ \sum_{j=1}^N y_j(M_1)^{c_1} \left. \frac{\cB_{\log \vec{y}(M_1)}(z_1,\ldots,z_N)}{\cB_{\log \vec{y}(M_1)}(\rho_N)} \right| \vec{y}(M_2) \right] \right]
\end{align*}
by commuting $\cD_{c_1}$ with the conditional expectation. Thus we obtain the right hand side of \eqref{eq:induction_observable} by consolidating the expectations.
\end{proof}

\begin{proof}[Proof of \Cref{thm:brownian_spherical_function}]
We prove the statement for $a = 0$, the general case follows from the identity
\[ \cB_{\boldsymbol{\eta} + a}(z_1,\ldots,z_N) = \left( \prod_{i=1}^N e^{a z_i} \right) \cB_{\boldsymbol{\eta}}(z_1,\ldots,z_N). \]
The density at time $t$ of $N$ Brownian bridges starting at $\vec{a}$ (at time $t = 0$), ending at $\vec{b}$ (at time $T$), and conditioned to never intersect is given by
\begin{align*}
\frac{1}{N! \det\left[ p_T(a_i,b_j) \right]_{i,j=1}^N} \det\left[ p_t(a_i, \eta_j) \right]_{i,j=1}^N \det \left[ p_{T-t}(\eta_i, b_j) \right]_{i,j=1}^N, \quad \quad p_t(x,y) = \frac{e^{-\frac{(x-y)^2}{2t}}}{\sqrt{2\pi t}}
\end{align*}
which expands out to
\begin{align*}
\frac{1}{(2\pi t(1-\frac{t}{T}))^{N/2}} \left( \prod_{i=1}^N e^{-\frac{a_i^2}{2t} + \frac{a_i^2}{2T} - \frac{b_i^2}{2(T-t)} + \frac{b_i^2}{2T}} e^{-\frac{T\eta_i^2}{2t(T-t)}} \right) \frac{\det\left[e^{\frac{a_i\eta_j}{t}} \right]_{i,j=1}^N \det\left[ e^{\frac{b_i\eta_j}{T-t}} \right]_{i,j=1}^N}{N! \det\left[ e^{\frac{a_ib_j}{T}} \right]_{i,j=1}^N}
\end{align*}
supported on $\boldsymbol{\eta} \in \R^N$, by e.g. \cite{Joh01}. Here, the density is on the unordered positions of the Brownian motions. Take $a_i = \e(N-i)$ and $b_i = T \mu_i$. Then the density becomes
\[ \frac{1}{N!(2\pi t(1-\frac{t}{T}))^{N/2}} \left( \prod_{i=1}^N e^{-\frac{\e^2 (T-t) (N-i)^2}{2tT} - \frac{tT\mu_i^2}{2(T-t)}} e^{-\frac{T\eta_i^2}{2t(T-t)}} \right) \det \left[ e^{\frac{T\mu_i \eta_j}{T-t}} \right]_{i,j=1}^N \prod_{1 \le i < j \le N} \frac{e^{\frac{\e \eta_i}{t}} - e^{\frac{\e \eta_j}{t}}}{e^{\e \mu_i} - e^{\e \mu_j}} \]
where we use the Vandermonde determinant identity
\[ \Delta(\vec{z}) = \prod_{1 \le i < j \le N} (z_i - z_j) = \det \left[ z_i^{N-j} \right]_{1 \le i,j \le N}. \]
Sending $\e \to 0$, then $T \to \infty$, we obtain
\[ \frac{1}{N!(2\pi t)^{N/2}} \left( \prod_{i=1}^N e^{-\frac{t\mu_i^2}{2}-\frac{\eta_i^2}{2t}} \right) \det\left[ e^{\mu_i \eta_j} \right]_{i,j=1}^N \frac{\Delta(\boldsymbol{\eta}/t)}{\Delta(\vec{\mu})}. \]
This is the time $t$ marginal density for Brownian motion on $\R^N$ starting at $\vec{0}$ with drift vector $\vec{\mu}$, more specifically this density corresponds to the unordered coordinates of this Brownian motion (so the density corresponds to a measure on $\R^N$ rather than on the Weyl chamber $\{x_1 \ge \cdots \ge x_N\}$). We have
\[ \frac{\cB_{\boldsymbol{\eta}}(\vec{z})}{\cB_{\boldsymbol{\eta}}(\vec{\mu})} = \frac{\det\left[ e^{z_i \eta_j} \right]_{i,j=1}^N}{\Delta(\vec{z})} \frac{\Delta(\vec{\mu})}{\det\left[e^{\mu_i \eta_j}\right]_{i,j=1}^N}. \]
Then
\begin{align*}
\E\left[ \frac{\cB_{\boldsymbol{\eta}(t)}(\vec{z})}{\cB_{\boldsymbol{\eta}(t)}(\vec{\mu})} \right] & = \frac{1}{N!(2\pi t)^{N/2}} \left( \prod_{i=1}^N e^{-\frac{t\mu_i^2}{2}} \right) \int_{\R^N} \det\left[ e^{z_i \eta_j} \right]_{i,j=1}^N \frac{\Delta(\boldsymbol{\eta}/t)}{\Delta(\vec{z})} \prod_{i=1}^N e^{-\frac{\eta_i^2}{2t}} \, d\boldsymbol{\eta} \\
& = \frac{1}{N!(2\pi t)^{N/2}} \left( \prod_{i=1}^N e^{-\frac{t\mu_i^2}{2}} \right) \frac{1}{\Delta(\vec{z})} \int_{\R^N} \det\left[ e^{z_i \eta_j} \right]_{i,j=1}^N \det \left[ \left( \frac{\eta_i}{t} \right)^{N-j} e^{-\frac{\eta_i^2}{2t}} \right]_{i,j=1}^N \, d\boldsymbol{\eta}
\end{align*}
By Andr\'eief's identity, we obtain
\begin{align} \label{eq:Andreief}
\begin{split}
\E\left[ \frac{\cB_{\boldsymbol{\eta}(t)}(\vec{z})}{\cB_{\boldsymbol{\eta}(t)}(\vec{\mu})} \right] &= \frac{1}{(2\pi t)^{N/2}} \left( \prod_{i=1}^N e^{-\frac{t\mu_i^2}{2}} \right) \frac{1}{\Delta(\vec{z})} \det \left[ \int_{\R} \left( \frac{x}{t} \right)^{N-j} e^{x z_i - \frac{x^2}{2t}} \, dx \right]_{i,j=1}^N \\
&= \frac{1}{(2\pi)^{N/2}} \left( \prod_{i=1}^N e^{\frac{t(z_i^2-\mu_i^2)}{2}} \right) \frac{1}{\Delta(\vec{z})} \det \left[ M_{N-j}(z_i) \right]_{i,j=1}^N
\end{split}
\end{align}
where
\[ M_n(z) := \int_{\R} \left(\frac{x}{t}\right)^n e^{-\frac{(x - tz)^2}{2t}} \, \frac{dx}{\sqrt{t}} = \sqrt{t} \int_\R x^n e^{-\frac{t(x - z)^2}{2}} \, dx. \]
We claim that $M_n(z)$ is a degree $n$ polynomial in $z$ with leading coefficient $\sqrt{2\pi}$. We proceed by induction on $n$. Clearly, $M_0(z) = \sqrt{2\pi}$. Observe that
\begin{align*}
M_n(0) = \sqrt{t} \int_{\R} x^n e^{-\frac{tx^2}{2}} \, dx = \sqrt{t} \int_{\R} (x - z)^n e^{-\frac{t(x-z)^2}{2}} \, dx = \sqrt{t} \sum_{k=0}^n \binom{n}{k} (-z)^{n-k} M_k(z).
\end{align*}
Rearranging, we get
\[ M_n(z) = t^{-1/2} M_n(0) - \sum_{k=0}^{n-1} \binom{n}{k} (-z)^{n-k} M_k(z) \]
By our induction hypothesis, we have
\[ M_n(z) = - \left( \sum_{k=0}^{n-1} \binom{n}{k} (-1)^{n-k} \right) \sqrt{2\pi} z^n + \mbox{lower degree terms} \]
Thus the top degree term is $\sqrt{2\pi} z^n$, completing the induction. Applying row operations, we have
\[ \det[M_{N-j}(z_i)]_{i,j=1}^N = (2\pi)^{N/2} \Delta(\vec{z}). \]
Plugging this into \eqref{eq:Andreief} completes the proof.
\end{proof}

\begin{proof}[Proof of \Cref{thm:brownian_observable}]
Set
\[ \Phi_t(z_1,\ldots,z_N) := \E \left[ \frac{\cB_{\boldsymbol{\xi}^{(N)}(\frac{t}{4}) - \frac{Nt}{2}}(z_1,\ldots,z_N)}{\cB_{\boldsymbol{\xi}^{(N)}(\frac{t}{4}) - \frac{Nt}{2}}(\rho_N)} \right]. \]
The joint distribution of $\boldsymbol{\xi}^{(N)}(\tfrac{t_k}{4}) - \tfrac{Nt_k}{2},\ldots,\boldsymbol{\xi}^{(N)}(\tfrac{t_1}{4}) - \tfrac{Nt_1}{2}$ is given by the joint distribution of the log squared singular values of
\begin{gather*}
\mathsf{X}^{(N,k)}(t_k), \\
\mathsf{X}^{(N,k-1)}(t_{k-1} - t_k) \mathsf{X}^{(N,k)}(t_k), \\
\vdots \\
\mathsf{X}^{(N,1)}(t_1 - t_2) \cdots \mathsf{X}^{(N,k-1)}(t_{k-1} - t_k) \mathsf{X}^{(N,k)}(t_k),
\end{gather*}
where $\mathsf{X}^{(N,1)}(t),\ldots,\mathsf{X}^{(N,k)}(t)$ are independent copies of $e^{-\frac{Nt}{4}} \mathsf{Y}^{(N)}(\tfrac{t}{4})$. Then \Cref{thm:observable_operators} implies
\[ \E \left[ \prod_{i=1}^k \sum_{j=1}^N e^{c_i\left(\xi_j^{(N)}(\frac{t_i}{4}) - \frac{Nt_i}{2}\right)} \right] = \left. \cD_{c_1} \Phi_{t_1 - t_2}(z_1,\ldots,z_N) \cdots \cD_{c_k} \Phi_{t_k - t_{k+1}}(z_1,\ldots,z_N) \right|_{\vec{z} = \rho_N}. \]

We compute $\Phi_t$. By \Cref{thm:brownian_nibm}, $\boldsymbol{\xi}^{(N)}(\frac{t}{4}) - \frac{Nt}{2}$ evolves as
\[ \vec{\eta}(t) - (N - \tfrac{1}{2})t \]
where $\vec{\eta}(s)$ is $N$ non-intersecting Brownian motions with drift $\rho_N = (N-1,N-2,\ldots,0)$, started at the origin. \Cref{thm:brownian_spherical_function} implies
\begin{align*}
\Phi_t(z_1,\ldots,z_N) &= \E \left[ \frac{\cB_{\vec{\eta}^{(N)}(t) - (N - \frac{1}{2})t}(z_1,\ldots,z_N)}{\cB_{\vec{\eta}^{(N)}(t) - (N - \frac{1}{2})t}(\rho_N)} \right] \\
&= \prod_{i=1}^N \frac{e^{\frac{t}{2}(z_i - N + \frac{1}{2})^2}}{e^{\frac{t}{2}(-i + \frac{1}{2})^2}}.
\end{align*}
Thus we have shown \eqref{eq:brownian_observable}.

We now show \eqref{eq:brownian_contour}. We first claim that if $f_1(z),\ldots,f_k(z)$ are entire functions, then (recalling $\cD_c$ acts on everything to the right of it in an expression)
\begin{align*}
& \cD_{c_1} \left( \prod_{i=1}^N f_1(z_i) \right) \cdots \cD_{c_k} \left( \prod_{i=1}^N f_k(z_i) \right) = \left( \prod_{i=1}^N f_1(z_i) \cdots f_k(z_i) \right) \\
& \quad \quad \times \frac{1}{(2\pi\bi)^k} \oint \cdots \oint \prod_{1 \le i < j \le k} \frac{(w_i - w_j)(w_i + c_i - w_j - c_j)}{(w_i - w_j - c_j)(w_i + c_i - w_j)} \prod_{i=1}^k \left( \prod_{\ell=i}^k \frac{f_\ell(w_i + c_i)}{f_\ell(w_i)} \right) \left( \prod_{j=1}^N \frac{w_i + c_i - z_j}{w_i - z_j} \right) \frac{dw_i}{c_i}
\end{align*}
where the $w_i$ contour is positively oriented around $z_1,\ldots,z_N$ for $1 \le i \le k$ and the $w_j$ contour contains $w_i + c_i$ and $w_i - c_j$ for $1 \le i < j \le k$. This can be proved by induction on $k$ using the residue theorem and the definition of $\cD_c$, see e.g. \cite[Appendix B]{Ahn19}.

If we set
\[ f_\ell(z) = e^{(t_\ell - t_{\ell+1})(z - N + \frac{1}{2})^2} \]
for $\ell = 1,\ldots,k$, and apply \eqref{eq:brownian_observable}, we obtain
\begin{align*}
\E \left[ \prod_{i=1}^k \sum_{j=1}^N e^{c_i\left(\xi_j^{(N)}(\frac{t_i}{4}) - \frac{Nt_i}{2}\right)} \right] =& \frac{1}{(2\pi\bi)^k} \oint \cdots \oint \prod_{1 \le i < j \le k} \frac{(w_i - w_j)(w_i + c_i - w_j - c_j)}{(w_i - w_j - c_j)(w_i + c_i - w_j)} \\
& \quad \quad \times \prod_{i=1}^k \left( \prod_{\ell=i}^k \frac{e^{\frac{(t_\ell - t_{\ell+1})}{2}(w_i + c_i - N + \frac{1}{2})^2}}{e^{\frac{(t_\ell - t_{\ell+1})}{2}(w_i - N + \frac{1}{2})^2}} \right) \left( \prod_{j=1}^N \frac{w_i + c_i - N + j}{w_i - N + j} \right) \frac{dw_i}{c_i}.
\end{align*}
Observe that
\[ \prod_{j=1}^N \frac{w + c_i - N + j}{w - N + j} = \frac{\Gamma(w_i + c_i + 1) \Gamma(w_i - N + 1)}{\Gamma(w_i + c_i - N + 1) \Gamma(w_i + 1)}. \]
By changing variables $w_i = z_i + N - 1$ and consolidating the product over $\ell$, \eqref{eq:brownian_contour} follows.
\end{proof}

\section{Limiting Line Ensemble} \label{sec:line_ensemble}

The purpose of this section is to introduce line ensembles introduced in \cite{CH14}, prove the existence of the limiting line ensemble $\vec{\xi}(t)$ and the convergence result \Cref{thm:dyson_limit}. We prove auxiliary lemmas on the way to the proof of \Cref{thm:dyson_limit} for later usage.

\begin{definition} \label{def:dyson_limit}
Let $\Sigma \subset \Z$ and $\Lambda \subset \R$ be intervals. Consider the topological space $C(\Sigma \times \Lambda)$ with the topology of uniform convergence on compact subsets of $\Sigma \times \Lambda$. We may view $C(\Sigma \times \Lambda)$ as the space $\Lambda \times C(\Lambda)$ of sequences $(\eta_i(t))_{i \in \Sigma}$ of continuous functions on $\Lambda$ by the identification $\eta(i,t) = \eta_i(t)$ for $\vec{\eta} \in C(\Sigma \times \Lambda)$. A \emph{line ensemble (on $\Lambda$)} is a probability measure on $C(\Sigma \times \Lambda)$ with respect to the Borel $\sigma$-algebra. For us, the set $\Sigma$ will always be $\{1,\ldots,k\}$ for some $k$ or $\Z_{>0}$. An \emph{infinite line ensemble} will then be a line ensemble with $\Sigma = \Z_{>0}$. A line ensemble $\vec{\eta}$ is \emph{non-intersecting} if $\eta_i(t) > \eta_j(t)$ for $i < j$ and $t \in \Lambda$ almost surely.
\end{definition}

\Cref{thm:dyson_limit} claims the existence of an infinite line ensemble $\{\vec{\xi}(t)\}_{t > 0}$ which is the limit of $\{\vec{\xi}^{(N)}(\tfrac{t}{4}) - \tfrac{Nt}{2} - \log N\}_{t > 0}$. The following theorem gives explicit expressions for certain observables of $\vec{\xi}(t)$.

\begin{theorem} \label{thm:line_ensemble}
We have:
\begin{enumerate}[(i)]
    \item For $c_1,\ldots,c_k > 0$,
    \begin{align} \label{eq:laplace_transform}
    \begin{split}
    & \E \left[ \prod_{i=1}^k \sum_{j=1}^\infty e^{c_i \xi_j(t_i)} \right] \\
    & \quad = \int \frac{dz_1}{2\pi\bi c_1} \cdots \int \frac{dz_k}{2\pi\bi c_k} \left( \prod_{1 \le i < j \le k} \frac{(z_i - z_j)(z_i + c_i - z_j - c_j)}{(z_i + c_i - z_j)(z_i - z_j - c_j)} \right) \prod_{i=1}^k \frac{e^{\frac{t_i}{2}(z_i + c_i - \frac{1}{2})^2}}{e^{\frac{t_i}{2} (z_i - \frac{1}{2})^2}}
    \frac{\Gamma(z_i)}{\Gamma(z_i + c_i)} 
    \end{split}
    \end{align}
    where the $z_i$ contour is an infinite contour positively oriented around $0,-1,-2,\ldots$ which starts at $-\infty - \bi \epsilon$ and ends at $-\infty + \bi \epsilon$ for $1 \le i \le k$, and the $z_j$ contour encloses $z_i + c_i$ and $z_i - c_j$ whenever $1 \le i < j \le k$.
    
    \item The spacetime correlation kernel for $\boldsymbol{\xi}(t)$ is given by
    \[ \rho_k(t_1,x_1;\ldots;t_k,x_k) = \det \left[ K(t_i,x_i;t_j,x_j) \right]_{1 \le i,j \le k} \]
    where
    \begin{align} \label{eq:correlation_kernel}
    K(s,x;t,y) = -\frac{1}{\sqrt{2\pi(t-s)}} e^{-\frac{(x-y)^2}{2(t-s)}} \1[t > s] + \int \frac{dz}{2\pi\bi} \int_{c - \bi \infty}^{c + \bi \infty} \frac{dw}{2\pi\bi} \frac{e^{\frac{tw^2}{2} - yw}}{e^{\frac{sz^2}{2} - xz}} \frac{1}{w - z} \frac{\Gamma(z + \frac{1}{2})}{\Gamma(w + \frac{1}{2})}.
    \end{align}
    and the $z$ contour is an infinite contour positively oriented around $-\tfrac{1}{2},-\tfrac{3}{2},-\tfrac{5}{2},\ldots$ which starts at $-\infty - \bi \epsilon$ and ends at $-\infty + \bi \epsilon$.
\end{enumerate}
\end{theorem}

\begin{remark}
The explicit expression for the correlation function will not be used directly in this paper, we only use the fact that it is determinantal. 
\end{remark}

The remainder of this section is devoted to the proofs of \Cref{thm:dyson_limit,thm:line_ensemble}. Our first step is to show the convergence of joint Laplace transforms and correlation functions.

\begin{proposition} \label{thm:brownian_convergence}
Fix $t_1 \ge \cdots \ge t_k > 0$. Suppose $\tau_1(N) \ge \cdots \ge \tau_k(N) > 0$ such that $t_i := \lim_{N\to\infty} \tau_i(N)$ for $1 \le i \le k$.
\begin{enumerate}[(i)]
    \item For any $c_1,\ldots,c_k > 0$,
    \begin{align*}
    & \lim_{N\to\infty} \E \left[ \prod_{i=1}^k \sum_{j=1}^N e^{c_i \left( \xi_j^{(N)}(\frac{\tau_i(N)}{4}) - \frac{N \tau_i(N)}{2} - \log N \right)} \right] \\
    & \quad \quad = \int \frac{dz_1}{2\pi\bi c_1} \cdots \int \frac{dz_k}{2\pi\bi c_k} \left( \prod_{1 \le i < j \le k} \frac{(z_i - z_j)(z_i + c_i - z_j - c_j)}{(z_i + c_i - z_j)(z_i - z_j - c_j)} \right) \prod_{i=1}^k \frac{e^{\frac{t_i}{2}(z_i + c_i - \frac{1}{2})^2}}{e^{\frac{t_i}{2} (z_i - \frac{1}{2})^2}} \frac{\Gamma(z_i)}{\Gamma(z_i + c_i)}
    \end{align*}
    where $c_1,\ldots,c_k > 0$, the $z_i$ contour is an infinite contour positively oriented around $0,-1,-2,\ldots$ which starts at $-\infty - \bi \epsilon$ and ends at $-\infty + \bi \epsilon$ for $1 \le i \le k$, and the $z_j$ contour encloses $z_i + c_i$ and $z_i - c_j$ whenever $1 \le i < j \le k$.
    
    \item Let $\rho^{(N)}_k(\tau_1,x_1;\ldots,\tau_k,x_k)$ denote the $k$th space-time correlation function of
    \[ \left( \xi_1^{(N)}\left(\frac{\tau}{4}\right) - \frac{N\tau}{2} - \log N, \ldots, \xi_N^{(N)}\left(\frac{\tau}{4}\right) - \frac{N\tau}{2} - \log N \right)_{\tau \ge 0}. \]
    Then
    \[ \lim_{N\to\infty} \rho^{(N)}_k(\tau_1(N),x_1;\ldots,\tau_k(N),x_k) = \det \left[ K(t_i,x_i;t_j,x_j) \right]_{1 \le i,j \le k} \]
    where $K(s,x;t,y)$ is given by \eqref{eq:correlation_kernel}.
\end{enumerate}
\end{proposition}

\begin{proof}[Proof of \Cref{thm:brownian_convergence}]
Let $c_1,\ldots,c_k > 0$. By \Cref{thm:brownian_observable}, we have
\begin{align*}
\begin{split}
& \E\left[ \prod_{i=1}^k \sum_{j=1}^N e^{c_i\left( \xi_j(\frac{\tau_i(N)}{4}) - \frac{N\tau_i(N)}{2} - \log N \right)} \right] = \left(\prod_{i=1}^k N^{-c_i} \right) \\
& \quad \quad \times \oint \frac{dz_1}{2\pi\bi c_1} \cdots \oint \frac{dz_k}{2\pi\bi c_k} \left( \prod_{1 \le i < j \le k} \frac{(z_i - z_j)(z_i + c_i - z_j - c_j)}{(z_i - z_j - c_j)(z_i + c_i - z_j)} \right) \prod_{i=1}^k \frac{e^{\frac{\tau_i(N)}{2}(z_i + c_i - \frac{1}{2})^2}}{e^{\frac{\tau_i(N)}{2}(z_i - \frac{1}{2})^2}} \frac{\Gamma(z_i + c_i + N) \Gamma(z_i)}{\Gamma(z_i + c_i) \Gamma(z_i + N)}
\end{split}
\end{align*}
where the $z_i$ contour is positively oriented around $0,-1,\ldots,-N+1$ for $1 \le i \le k$ and the $z_j$ contour contains $z_i + c_i$ and $z_i - c_j$ for $1 \le i < j \le k$. From Stirling's formula \cite[p141]{OLBC10} (see also \cite[Lemma 6.6]{Ahn19}) for the Gamma function, we have
\[ \frac{\Gamma(z_i + c_i + N)}{\Gamma(z_i + N)} = \frac{(z_i + c_i + N)^{z_i + c_i + N - \frac{1}{2}}}{(z_i + N)^{z_i + N - \frac{1}{2}}} e^{-c_i} (1 + O(1/N)) = N^{c_i}(1 + O(1/N)) \]
which holds uniformly on compact subsets of the $z_i$ contour. Combining this with the decay of the integrand for $\Re z \ll 0$ but $\Re z > -N-1$ and $|\Im z|$ bounded away from $0$, we obtain the desired expression
\[ \int \frac{dz_1}{2\pi\bi c_1} \cdots \int \frac{dz_k}{2\pi\bi c_k} \left( \prod_{1 \le i < j \le k} \frac{(z_i - z_j)(z_i + c_i - z_j - c_j)}{(z_i + c_i - z_j)(z_i - z_j - c_j)} \right) \prod_{i=1}^k \frac{e^{\frac{t_i}{2}(z_i + c_i - \frac{1}{2})^2}}{e^{\frac{t_i}{2} (z_i - \frac{1}{2})^2}} \frac{\Gamma(z_i)}{\Gamma(z_i + c_i)} \]
in the limit as $N\to\infty$. Note that the decay of the exponential terms at infinity along the contour is clear. To see the decay of the gamma quotient, we may use the reflection formula for the Gamma function
\[ \Gamma(z)\Gamma(1 - z) = \frac{\pi}{\sin(\pi z)}. \]

Recalling \Cref{thm:brownian_nibm}, we can explicitly write down the spacetime correlation kernel for $\{\vec{\xi}^{(N)}(\frac{t}{4}) - \frac{Nt}{2}\}_{t > 0}$ by \cite{Joh01} (see also \cite[Proposition 4.1]{CP16})
\begin{align*}
K_N(s,x;t,y) = -\frac{1}{\sqrt{2\pi(t-s)}} e^{-\frac{(x-y)^2}{2(t-s)}} \1[t > s] + \oint_\gamma \frac{dz}{2\pi\bi} \int_{\Gamma_c} \frac{dw}{2\pi\bi} \frac{e^{\frac{tw^2}{2} - yw}}{e^{\frac{sz^2}{2} - xz}} \frac{1}{w - z} \prod_{i=1}^N \frac{w + i - \frac{1}{2}}{z + i - \frac{1}{2}}
\end{align*}
where $\gamma$ is a simple closed curve positively oriented around $\{-i+\tfrac{1}{2}\}_{i=1}^N$ and $\Gamma_c:\tau \mapsto c + \bi \tau, \tau \in \R$ such that $\gamma$ and $\Gamma_c$ are disjoint. Thus
\[ \rho_k^{(N)}(\tau_1,x_1;\ldots;\tau_k,x_k) = \det \left[ K_N(\tau_i,x_i + \log N;\tau_j,x_j + \log N) \right]_{i,j=1}^k. \]
We can write
\begin{align*}
K_N(s,x+\log N;t,y+\log N) &= -\frac{1}{\sqrt{2\pi(t-s)}} e^{-\frac{(x-y)^2}{2(t-s)}} \1[t > s] \\
& \quad \quad + \oint_\gamma \frac{dz}{2\pi\bi} \int_{\Gamma_c} \frac{dw}{2\pi\bi} \frac{e^{\frac{tw^2}{2} - yw}}{e^{\frac{sz^2}{2} - xz}} \frac{N^{z-w}}{w - z} \frac{\Gamma(w + N + \frac{1}{2})}{\Gamma(z + N + \frac{1}{2})} \frac{\Gamma(z + \frac{1}{2})}{\Gamma(w + \frac{1}{2})}. \end{align*}
From Stirling's formula for the Gamma function as before, we find
\[ \frac{\Gamma(w+N+\frac{1}{2})}{\Gamma(z+N+\frac{1}{2})} = \frac{(w+N+\frac{1}{2})^{w+N}}{(z+N+\frac{1}{2})^{z+N}} e^{z - w} (1 + O(1/N)) = N^{w - z} (1 + O(1/N)). \]
Thus, we have
\begin{align*}
& \lim_{N\to\infty} K_N(s,x + \log N;t,y + \log N) \\
& \quad \quad = -\frac{1}{\sqrt{2\pi(t-s)}} e^{-\frac{(x-y)^2}{2(t-s)}} \1[t > s] + \int \frac{dz}{2\pi\bi} \int_{\Gamma_c} \frac{dw}{2\pi\bi} \frac{e^{\frac{tw^2}{2} - yw}}{e^{\frac{sz^2}{2} - xz}} \frac{1}{w - z} \frac{\Gamma(z+\frac{1}{2})}{\Gamma(w+\frac{1}{2})}
\end{align*}
where the $z$ contour is an infinite contour positively oriented around $-\tfrac{1}{2},-\tfrac{3}{2},-\tfrac{5}{2},\ldots$, starting at $-\infty- \bi \epsilon$ and ending at $-\infty+ \bi \epsilon$. For full rigor, we must control the tail of the $z$-contour for $\Re z \ll 0$. This is managed by the reflection formula for the gamma function and the $e^{-\frac{sz^2}{2}}$ term, as before.
\end{proof}

The next two lemmas are the key to proving \Cref{thm:dyson_limit,thm:line_ensemble}. They are stated in a manner convenient for later usage. The first lemma establishes the existence of a limiting process.

\begin{lemma} \label{thm:existence}
There exists a process $\{\vec{\xi}(t) := (\xi_1(t),\xi_2(t),\ldots) \}_{t > 0}$ with joint Laplace transform given by \eqref{eq:laplace_transform} and spacetime correlation kernel given by \eqref{eq:correlation_kernel}.
\end{lemma}

The next lemma links convergence of Laplace transforms with convergence in finite dimensional distributions.

\begin{lemma} \label{thm:laplace_implies_findim}
Fix $t_1 \ge \cdots \ge t_k > 0$. Let $\tau_1(N) \ge \cdots \ge \tau_k(N) > 0$ such that $t_i := \lim_{N\to\infty} \tau_i(N)$ for $1 \le i \le k$. Suppose $\left\{(\mathsf{y}_1^{(N)}(\tau) \ge \cdots \ge \mathsf{y}_N^{(N)}(\tau)\right\}_{\tau > 0}$ is a random $\R^N$-valued process such that
\begin{align} \label{eq:assum_laplace_convergence}
\lim_{N\to\infty} \E \left[ \prod_{i=1}^k \sum_{j=1}^N e^{c_i \mathsf{y}_j^{(N)}(\tau_i(N))} \right] = \E \left[ \prod_{i=1}^k \sum_{j=1}^\infty e^{c_i \xi_j(t_i)} \right]
\end{align}
for $0 < c_1,\ldots,c_k \le \e$ and some $\e > 0$ (which may vary with $k$). Then
\[ \lim_{N\to\infty} \PP\left( \mathsf{y}_j(\tau_i(N)) \le a_{i,j}: 1 \le i \le k, 1 \le j \le h \right) = \PP\left( \xi_j(t_i) \le a_{i,j}: 1 \le i \le k, 1 \le j \le h \right) \] 
for any real numbers $a_{i,j}$ ($1 \le i \le k$, $1 \le j \le h$) and any positive integer $k$.
\end{lemma}

\begin{proof}[Proof of \Cref{thm:existence,thm:laplace_implies_findim}]
The argument below closely follows the ideas from \cite[Section 5]{Sos99} and \cite[Section 4.1.3]{Oko00} to show that the convergence of Laplace transforms of the correlation functions implies the desired convergence in finite dimensional distributions. Let $\rho^{(N)}_k(\tau_1,x_1;\ldots;\tau_k,x_k)$ denote the space-time correlation function for the process
\[ \left(\mathsf{y}_1^{(N)}(\tau),\ldots,\mathsf{y}_k^{(N)}(\tau) \right). \]
Our assumption \eqref{eq:assum_laplace_convergence} implies the existence of the limits
\begin{align} \label{eq:limits_correlation_functions}
\lim_{N\to\infty} \int_{\R^k} e^{c_1x_1 + \cdots + c_k x_k} \rho^{(N)}_k(\tau_1(N),x_1;\ldots;\tau_k(N),x_k) dx_1 \cdots dx_k
\end{align}
for $0 < c_1,\ldots,c_k < \e$ where the limit is given by a finite linear combination of \eqref{eq:laplace_transform}. We want to show that this limit is given by some limiting measure $\rho_k(t_1,x_1;\ldots;t_k,x_k)$. For this, define the measure
\[ \varrho^{(N)}_k(x_1,\ldots,x_k)dx_1 \cdots dx_k := e^{\theta x_1 + \cdots + \theta x_k} \rho^{(N)}_k(\tau_1(N),x_1;\ldots;\tau_k(N),x_k) dx_1 \cdots dx_k \]
where $\theta = \e/2$. The existence and form of the limits \eqref{eq:limits_correlation_functions} implies the weak convergence of $\varrho^{(N)}_k$ to some limiting finite measure $\varrho_k$ as measures on $\R^k$ where the latter measure is finite. Define $\rho_k$ by
\[ \rho_k(t_1,x_1;\ldots;t_k,x_k) dx_1 \cdots dx_k:= e^{-\theta x_1 - \cdots - \theta x_k} \varrho_k(x_1,\ldots,x_k) dx_1 \cdots dx_k \]
where we note the suppression of the dependence on the $\tau$'s and $t$'s in the notation for $\varrho^{(N)}_k$ and $\varrho_k$. Thus
\[ \rho^{(N)}_k(\tau_1(N),x_1;\ldots;\tau_k(N),x_k) dx_1 \cdots dx_k \to \rho_k(t_1,x_1;\ldots;t_k,x_k) dx_1 \cdots dx_k \]
weakly on $\R^k$. By \Cref{thm:brownian_convergence}, this convergence holds in particular for $\mathsf{y}^{(N)}(\tau) = \vec{\xi}^{(N)}(\frac{\tau}{4}) - \frac{N\tau}{2} - \log N$ so that
\[ \rho_k(t_1,x_1;\ldots,t_k,x_k) = \det\left[ K(t_i,x_i;t_j,x_j) \right]_{1 \le i,j \le k} \]
where $K(s,x;t,y)$ is given by \eqref{eq:correlation_kernel}.

The weak convergence of the correlation functions implies that the joint moments of random variables of the form
\[ \cY_{\tau_i(N)}^{(N)}(S) := |\{j: \mathsf{y}_j^{(N)}(\tau_i(N)) \in S \}|, \quad \quad S \subset [c,\infty), \quad \quad 1 \le i \le k, \quad \quad c > 0  \]
converge to corresponding joint moments of some limiting random variables
\[ \cY_{t_i}(S), \quad \quad S \subset [c,\infty), \quad \quad 1 \le i \le k, \quad \quad c > 0. \]

Since the limit $\rho_k$ is determinantal, the joint moments of the $\cY_{t_i}(S)$ do not grow faster than factorials so that the convergence of joint moments implies convergence in distribution. Therefore the probabilities
\[ \PP\left(\mathsf{y}_j^{(N)}(\tau_i(N)) \le a_{i,j}: 1 \le i \le k, 1 \le j \le h\right) \]
converge as $N\to\infty$ as they can be expressed as a finite linear combination of probabilities of the form
\[ \PP\left( \cY_{\tau_1(N)}^{(N)}(S_{1,1}) = n_{1,1},\ldots \cY_{\tau_1(N)}^{(N)}(S_{1,r_1}) = n_{1,r_1}, \ldots, \cY_{\tau_k(N)}^{(N)}(S_{k,1}) = n_{k,1},\ldots,\cY_{\tau_k(N)}^{(N)}(S_{k,r_k}) = n_{k,r_k} \right), \]
where the sets $S_{i,r}$ are among $(a_{i,1},\infty), (a_{i,2},a_{i,1}],\ldots,(a_{i,h},a_{i,h-1}]$. This proves the existence of the limit (in finite dimensional distributions) process $\{(\xi_1,\xi_2,\ldots)\}_{t > 0}$, where the Laplace transform and spacetime correlation kernel are necessarily given by \eqref{eq:laplace_transform} and \eqref{eq:correlation_kernel}. Thus \Cref{thm:existence,thm:laplace_implies_findim} follow.
\end{proof}

\begin{proof}[Proof of \Cref{thm:dyson_limit} and \Cref{thm:line_ensemble}]
We want to upgrade the convergence in finite dimensional distributions of
\[ \left( \xi_1^{(N)}(\tfrac{t}{4}) - \tfrac{Nt}{2} - \log N,\ldots, \xi_N^{(N)}(\tfrac{t}{4}) - \tfrac{Nt}{2} - \log N \right) \]
implied by \Cref{thm:brownian_convergence} and \Cref{thm:laplace_implies_findim} to the stronger notion of convergence of line ensembles for \Cref{thm:dyson_limit}. The machinery for this is supplied by \cite[Proposition 3.6]{CH14}. We can argue as in \cite[Proposition 3.12]{CH14} to check that our line ensembles satisfy the hypotheses of \cite[Proposition 3.6]{CH14}, using the determinantal structure of the line ensembles from \Cref{thm:brownian_convergence}. The statements in \cite{CH14} are for line ensembles on $[-T,T]$, so minor modifications in the statement of hypotheses need to be made to obtain the convergence of our line ensembles on $[\tfrac{1}{T},T]$. \Cref{thm:line_ensemble} follows from \Cref{thm:existence,thm:laplace_implies_findim}.
\end{proof}

\section{The \texorpdfstring{$S$}{S}-Transform and \texorpdfstring{$\psi$}{Psi Function}} \label{sec:S-transform}

Given a probability measure $\mu$ on $\R_{\ge 0}$, we can define its $\psi$-function and $S$-transform. The former is a generating function for the moments of $\mu$ and the latter plays the role of the log characteristic function from classical probability in the context of free probability, where the multiplicative free convolution corresponds to summation of independent random variables, see e.g. \cite{V87,BV92}. We collect several properties of these functions for the analysis in subsequent sections.

\begin{definition}
Given a probability measure $\mu$ supported in $\R_{\ge 0}$, let
\[ \psi_\mu(z) := \int \frac{zx}{1 - zx} \, d\mu(x), \quad \quad z \in \C \setminus \supp \mu. \]
\end{definition}

\begin{definition}
Let $\cM$ denote the set of compactly supported Borel probability measures on $\R_{> 0}$, in particular $\inf \supp \mu > 0$ for $\mu \in \cM$. We view $\cM$ as a topological space under the weak topology. Given a closed interval $I \subset \R_{>0}$, let $\cM_I \subset \cM$ denote the subset of probability measures supported in $I$, which is compact under the weak topology.
\end{definition}

Assume that $\mu \in \cM$. Then $\psi_\mu$ is analytic on $(\C \cup \{\infty\}) \setminus J$ where $J$ is some bounded interval in $\R_{>0}$ which contains $\{x^{-1}: \supp \mu \}$. Moreover,
\[ \psi_\mu'(z) = \int \frac{x}{(1 - zx)^2} \, d\mu(x) \]
which is positive for $z \le 0$. Thus there exists an inverse $\psi_\mu^{-1}$ defined in a neighborhood of $[-1,0]$ which is meromorphic with a simple pole at $-1$ and a zero at $0$.

\begin{definition}
The \emph{$S$-transform} of $\mu \in \cM$ is given by
\[ S_\mu(u) := \frac{1+u}{u} \psi_\mu^{-1}(u). \]
\end{definition}

In view of the discussion above, for $\mu$ compactly supported in $\R_{>0}$, the $S$-transform is defined in a neighborhood of $[-1,0]$. From our discussion, we see that the following properties hold:

\begin{proposition} \label{thm:psi}
Fix a compact subset $I \subset \R_{>0}$. Then there exists a neighborhood $U \subset \C$ of $[-1,0]$ such that for all $\mu \in \cM_I$
\begin{enumerate}[(i)]
    \item $\psi_\mu^{-1}(z)$ is well-defined, bijective, meromorphic function on $U$ with a unique pole at $-1$ and zero at $0$;
    
    \item $S_\mu(z)$ is holomorphic with no zeros on $U$, and
    
    \item the maps $\mu \mapsto S_\mu$ and $\mu \mapsto \psi_\mu^{-1}$ on $\cM_I$ are continuous where the topology of the images are with respect to uniform convergence on compact subsets of $U$.
\end{enumerate}
\end{proposition}

Here are some properties of the $S$-transform which follow from \cite[Proposition 3.1]{BV92}:

\begin{proposition}
\begin{enumerate}[(i)]
    \item $S_\mu'(u) \le 0$ for $u \in [-1,0]$.
    \item $S_\mu(u) > 0$ for $u \in [-1,0]$.
    \item $\overline{S_\mu(u)} = S_\mu(\overline{u})$.
\end{enumerate}
\end{proposition}

We record a lemma which evaluates the $S$-transform and its first and second derivatives at $0$.

\begin{lemma} \label{thm:S_and_S'}
Suppose $\mu \in \cM$. Let
\[ \kappa_1(\mu) := \int x \, d\mu(x), \quad \quad \kappa_2(\mu) := \int x^2 \, d\mu(x) - \left( \int x \, d\mu(x) \right)^2 \]
denote the mean and variance of $\mu$ respectively. Then
\begin{gather*}
S_\mu(0) = \frac{1}{\kappa_1(\mu)}, \quad \quad S_\mu'(0) = -\frac{\kappa_2(\mu)}{\kappa_1(\mu)^3}, \\
S_\mu''(0) = 4 \frac{\left( \int x^2 \, d\mu(x) \right)^2}{\left( \int x \, d\mu(x) \right)^5} - 2 \frac{\int x^3 \, d\mu(x)}{\left( \int x \, d\mu(x) \right)^4} - 2 \frac{\int x^2 \, d\mu(x)}{\left( \int x \, d\mu(x) \right)^3}.
\end{gather*}
\end{lemma}

\begin{proof}
From the expansion
\[ \psi_\mu(z) = z \int x \, d\mu(x) + z^2 \int x^2 \, d\mu(x) + z^3 \int x^3 \, d\mu(x) + O(|z|^4), \quad \quad |z| \to 0, \]
we get
\[ \psi_\mu^{-1}(u) = u \frac{1}{\int x \, d\mu(x)} - u^2 \frac{\int x^2 \, d\mu(x)}{ \left( \int x \, d\mu(x) \right)^3} + u^3 \left(2 \frac{\int x^2 \, d\mu(x)}{\left( \int x \, d\mu(x) \right)^5} - \frac{\int x^3 \, d\mu(x)}{\left( \int x \, d\mu(x) \right)^4} \right) + O(|u|^4) , \quad \quad |u| \to 0 \]
so that
\begin{align*}
S_\mu(u) =& \frac{1}{\int x \, d\mu(x)} + \left( \frac{1}{\int x \, d\mu(x)} - \frac{\int x^2 \, d\mu(x)}{\left(\int x \, d\mu(x)\right)^3} \right) u \\
& \quad \quad + \left( 2 \frac{\left( \int x^2 \, d\mu(x) \right)^2}{\left( \int x \, d\mu(x) \right)^5} - \frac{\int x^3 \, d\mu(x)}{\left( \int x \, d\mu(x) \right)^4} - \frac{\int x^2 \, d\mu(x)}{\left( \int x \, d\mu(x) \right)^3} \right) u^2 + O(|u|^3)
\end{align*}
as $|u| \to 0$. The result follows.
\end{proof}

We conclude this section with a lemma on ratios of Cauchy determinants involving the $\psi$-function, for later use.

\begin{lemma} \label{thm:psi_cauchy}
Fix a compact subset $I \subset \R_{>0}$ and a positive integer $k$. Then there exists a neighborhood $U \subset \C$ of $[-1,0]$ such that for all $\mu \in \cM_I$ and $u_1,\ldots,u_k, v_1,\ldots,v_k \in U$, the bound
\begin{align} \label{eq:psi_cauchy_bound}
C^{-1} < \left| \frac{\det\left( \frac{1}{\psi_\mu^{-1}(u_i) - \psi_\mu^{-1}(v_j)} \right)_{1 \le i,j \le k}}{\det\left( \frac{1}{u_i - v_j} \right)_{1 \le i,j \le k}} \prod_{i=1}^k \frac{1}{\sqrt{\psi_\mu'(\psi_\mu^{-1}(u_i)) \psi_\mu'(\psi_\mu^{-1}(v_i))}} \right| < C
\end{align}
holds for some constant $C > 0$ independent of $\mu \in \cM_I$. Moreover,
\begin{align} \label{eq:psi_cauchy_estimate}
\frac{\det\left( \frac{1}{\psi_\mu^{-1}(u_i) - \psi_\mu^{-1}(v_j)} \right)_{1 \le i,j \le k}}{\det\left( \frac{1}{u_i - v_j} \right)_{1 \le i,j \le k}} \prod_{i=1}^k \frac{1}{\sqrt{\psi_\mu'(\psi_\mu^{-1}(u_i)) \psi_\mu'(\psi_\mu^{-1}(v_i))}} = 1 + O\left( \max_{1 \le i \le k} |u_i - v_i|^2 \right)
\end{align}
uniformly over $\mu \in \cM_I$, $u_1,\ldots,u_k,v_1,\ldots,v_k \in U$.
\end{lemma}

\begin{remark}
From \Cref{thm:psi}, $\psi_\mu'(\psi_\mu^{-1}(u))$ is nonzero for $u$ in a neighborhood $U$ of $[-1,0]$ and positive on $[-1,0]$. Therefore, the square root is well-defined, where we take the standard branch for $u \in [-1,0]$ and extend by continuity on $U$.
\end{remark}

\begin{proof}[Proof of \Cref{thm:psi_cauchy}]
Our starting point is a proof of the case $k = 1$, restated in the following claim:

\begin{claim} \label{claim:psi_difference}
Fix a compact subset $I \subset \R_{>0}$. Then there exists a neighborhood $U \subset \C$ of $[-1,0]$ such that for all $\mu \in \cM_I$ and $u, v \in U$, we have
\begin{align} \label{eq:psi_diff_bounded_nonvanishing}
C^{-1} < \frac{1}{\psi_\mu^{-1}(u) - \psi_\mu^{-1}(v)} \frac{u - v}{\sqrt{\psi_\mu'(\psi_\mu^{-1}(u)) \psi_\mu'(\psi_\mu^{-1}(v))}} < C
\end{align}
for some constant $C$ independent of $\mu \in \cM_I$. Moreover,
\begin{align} \label{eq:psi_diff_close_points}
\frac{1}{\psi_\mu^{-1}(u) - \psi_\mu^{-1}(v)} \frac{u - v}{\sqrt{\psi_\mu'(\psi_\mu^{-1}(u)) \psi_\mu'(\psi_\mu^{-1}(v))}} = 1 + O(|u-v|^2) \end{align}
uniformly over $\mu \in \cM_I$ and $u,v \in U$.
\end{claim}
\begin{proof}[Proof of \Cref{claim:psi_difference}]
Choose $U \supset [-1,0]$ so that $\psi_\mu^{-1}$ is analytic on its closure for every $\mu \in \cM_I$, where existence is guaranteed by \Cref{thm:psi}. Observe that
\[ \mathrm{C}(u,v) := \frac{1}{\psi_\mu^{-1}(u) - \psi_\mu^{-1}(v)} \frac{u - v}{\sqrt{\psi_\mu'(\psi_\mu^{-1}(u)) \psi_\mu'(\psi_\mu^{-1}(v))}} \]
and its reciprocal have no poles of codimension $1$ and are thus holomorphic on $\cl(U)^2$ by Riemann's second extension theorem \cite[Theorem 7.1.2]{GR84}, as in \cite[Proof of Lemma 3.5]{GS}. Therefore $\mathrm{C}(u,v)$ is bounded and does not vanish on $U$. This implies \eqref{eq:psi_diff_bounded_nonvanishing} where the uniformity of $C$ follows from the compactness of $\cM_I$ and $\cl(U)$, and the continuity of $\mathrm{C}(u,v)$ as a function of $\mu$, $u$, and $v$.

It remains to show \eqref{eq:psi_diff_close_points}. Assume without loss of generality that $I = [a^{-1},a]$ for some $a > 1$. Fix $\delta > 0$ small and let $W_\delta := \{ w \in U: |w+1| \ge \delta\}$.

We start by showing \eqref{eq:psi_diff_close_points} for $u,v \in W_\delta$. Assuming $u,v \in W_\delta$, since
\[ \psi_\mu^{-1}(u) - \psi_\mu^{-1}(v) = \frac{1}{\psi_\mu'(\psi_\mu^{-1}(v))} (u-v) - \frac{1}{2} \frac{\psi_\mu''(\psi_\mu^{-1}(v))}{\psi_\mu'(\psi_\mu^{-1}(v))^3} (u-v)^2 + O(|u-v|^3) \]
we have
\begin{align*}
& \frac{1}{\psi_\mu^{-1}(u) - \psi_\mu^{-1}(v)} \frac{u - v}{\sqrt{\psi_\mu'(\psi_\mu^{-1}(u)) \psi_\mu'(\psi_\mu^{-1}(v))}} \\
& \quad \quad = \frac{\sqrt{\psi_\mu'(\psi_\mu^{-1}(v))}}{\sqrt{\psi_\mu'(\psi_\mu^{-1}(u))}} \frac{1}{1 - \frac{1}{2} \frac{\psi_\mu''(\psi_\mu^{-1}(v))}{\psi_\mu'(\psi_\mu^{-1}(v))^2} (u-v) + O(|u-v|^2)} \\
& \quad \quad = \frac{\sqrt{\psi_\mu'(\psi_\mu^{-1}(v))}}{\sqrt{\psi_\mu'(\psi_\mu^{-1}(u))}} \left( 1 + \frac{1}{2} \frac{\psi_\mu''(\psi_\mu^{-1}(v))}{\psi_\mu'(\psi_\mu^{-1}(v))^2} (u-v) + O(|u-v|^2) \right).
\end{align*}
Since
\[ \log \psi_\mu'(\psi_\mu^{-1}(u)) = \log \psi_\mu'(\psi_\mu^{-1}(v)) + \frac{\psi_\mu''(\psi_\mu^{-1}(v))}{\psi_\mu'(\psi_\mu^{-1}(v))^2}(u - v) + O(|u-v|^2), \]
we have
\begin{align*}
\frac{\sqrt{\psi_\mu'(\psi_\mu^{-1}(v))}}{\sqrt{\psi_\mu'(\psi_\mu^{-1}(u))}} &= \exp\left( \frac{1}{2} \log \psi_\mu'(\psi_\mu^{-1}(v)) - \frac{1}{2} \log \psi_\mu'(\psi_\mu^{-1}(u)) \right) \\
&= \exp\left( -\frac{1}{2} \frac{\psi_\mu''(\psi_\mu^{-1}(v))}{\psi_\mu'(\psi_\mu^{-1}(v))^2}(u-v) + O(|u-v|^2) \right) \\
&= 1 - \frac{1}{2} \frac{\psi_\mu''(\psi_\mu^{-1}(v))}{\psi_\mu'(\psi_\mu^{-1}(v))^2}(u-v) + O(|u-v|^2).
\end{align*}
Combining these estimates proves \eqref{eq:psi_diff_close_points} holds for $u,v \in W_\delta$.

To complete the proof, we show that \eqref{eq:psi_diff_close_points} hold for $u,v \in \{w\in U: |w| \ge \delta\} = -(W_\delta + 1)$. For $\delta$ sufficiently small, $W_\delta$ and $-(W_\delta + 1)$ cover $U$. This is sufficient since the estimate \eqref{eq:psi_diff_close_points} holds trivially if $u,v$ are separated. We prove this by reduction to the case for $W_\delta$. We may write
\[ \frac{1}{\psi_\mu^{-1}(u) - \psi_\mu^{-1}(v)} \frac{u - v}{\sqrt{\psi_\mu'(\psi_\mu^{-1}(u)) \psi_\mu'(\psi_\mu^{-1}(v))}} = \frac{1}{\psi_\nu^{-1}(\wt{u}) - \psi_\nu^{-1}(\wt{v})} \frac{\wt{u} - \wt{v}}{\sqrt{\psi_\nu'(\psi_\nu^{-1}(\wt{u})) \psi_\nu'(\psi_\nu^{-1}(\wt{v}))}} \]
where $\wt{u} = -(u+1)$, $\wt{v} = -(v+1)$, and let $\nu$ denote the Borel probability measure on $\R_{>0}$ determined by $\nu([c_1,c_2]) = \mu([c_2^{-1},c_1^{-1}])$ for any $0 < c_1 \le c_2 < \infty$. Indeed, observe
\begin{align*}
\psi_\nu(z) &= -\psi_\mu(z^{-1}) - 1 \\
\psi_\nu^{-1}(w) &= \psi_\mu^{-1}(-(w+1))^{-1} \\
\psi_\nu'(z) &= \frac{1}{z^2} \psi_\mu'(z^{-1}) \\
\psi_\nu'(\psi_\nu^{-1}(w)) &= \psi_\mu^{-1}(-(w+1))^2 \psi_\mu'(\psi_\mu^{-1}(-(w+1))).
\end{align*}
Since $\nu \in \cM_I$ (recall $I = [a^{-1},a]$), this completes the proof.
\end{proof}

By the Cauchy determinant formula, which states
\[ \det\left( \frac{1}{a_i - b_j} \right)_{1 \le i,j \le k} = \frac{\prod_{1 \le i < j \le k} (a_i - a_j)(b_j - b_i)}{\prod_{i,j=1}^k (a_i - b_j)}, \]
we have
\begin{align*}
& \frac{\det\left( \frac{1}{\psi_\mu^{-1}(u_i) - \psi_\mu^{-1}(v_j)} \right)_{1 \le i,j \le k}}{\det\left( \frac{1}{u_i - v_j} \right)_{1 \le i,j \le k}} \prod_{i=1}^k \frac{1}{\sqrt{\psi_\mu'(\psi_\mu^{-1}(u_i)) \psi_\mu'(\psi_\mu^{-1}(v_i))}} \\
& \quad \quad = \prod_{i=1}^k \frac{u_i - v_i}{\psi_\mu^{-1}(u_i) - \psi_\mu^{-1}(v_i)}  \frac{1}{\sqrt{\psi_\mu'(\psi_\mu^{-1}(u_i)) \psi_\mu'(\psi_\mu^{-1}(v_i))}} \\
& \quad \quad \quad \quad \times\prod_{1 \le i < j \le k} \frac{u_i - v_j}{\psi_\mu^{-1}(u_i) - \psi_\mu^{-1}(v_j)} \frac{v_i - u_j}{\psi_\mu^{-1}(v_i) - \psi_\mu^{-1}(u_j)} \frac{\psi_\mu^{-1}(u_i) - \psi_\mu^{-1}(u_j)}{u_i - u_j} \frac{\psi_\mu^{-1}(v_i) - \psi_\mu^{-1}(v_j)}{v_i - v_j}.
\end{align*}
Setting
\[ \mathrm{C}(u,v) := \frac{u - v}{\psi_\mu^{-1}(u) - \psi_\mu^{-1}(v)} \frac{1}{\sqrt{\psi_\mu'(\psi_\mu^{-1}(u)) \psi_\mu'(\psi_\mu^{-1}(v))}}, \]
we obtain
\[ \frac{\det\left( \frac{1}{\psi_\mu^{-1}(u_i) - \psi_\mu^{-1}(v_j)} \right)_{1 \le i,j \le k}}{\det\left( \frac{1}{u_i - v_j} \right)_{1 \le i,j \le k}} \prod_{i=1}^k \frac{1}{\sqrt{\psi_\mu'(\psi_\mu^{-1}(u_i)) \psi_\mu'(\psi_\mu^{-1}(v_i))}} = \prod_{i=1}^k \mathrm{C}(u_i,v_i) \prod_{1 \le i < j \le k} \frac{\mathrm{C}(u_i,v_j) \mathrm{C}(v_i,u_j)}{\mathrm{C}(u_i,u_j) \mathrm{C}(v_i,v_j)}. \]
Then \Cref{claim:psi_difference} implies the bound \eqref{eq:psi_cauchy_bound}.

For the estimate \eqref{eq:psi_cauchy_estimate}, first note that
\[ \mathrm{C}(u_i,v_i) = 1 + O(|u_i - v_i|^2) \]
by \Cref{claim:psi_difference}, and
\[ \frac{\mathrm{C}(u_i,v_j) \mathrm{C}(v_i,u_j)}{\mathrm{C}(u_i,u_j) \mathrm{C}(v_i,v_j)} = 1 + O\left(\max(|u_i - v_i|^2,|u_j - v_j|^2)\right), \]
which can be seen by Taylor expanding in $u_i$ near $v_i$ and $u_j$ near $v_j$.
\end{proof}

\section{Asymptotics of Multivariate Bessel Functions} \label{sec:asymptotics_bessel}

Given $v_1,\ldots,v_k \in \{N-1,N-2,\ldots,0\}$, define
\[ B_\mu^{(N)}(u_1,\ldots,u_k; v_1,\ldots,v_k) := \frac{\cB_{\vec{a}}(u_1,\ldots,u_k,N-1,\ldots,\wh{v_1},\ldots,\wh{v_k},\ldots,0)}{\cB_{\vec{a}}(N-1,\ldots,0)} \]
where $\mu := \frac{1}{N} \sum_{i=1}^N \delta_{e^{a_i}}$ and the hat notation means that $v_1,\ldots,v_k$ are omitted from $N-1,N-2,\ldots,0$. In other words, the multivariate Bessel function in the numerator takes as input $\rho_N$ with $v_1,\ldots,v_k$ replaced by $u_1,\ldots,u_k$. In this section, we obtain asymptotics for these normalized multivariate Bessel functions in preparation for proving \Cref{thm:main}. We note that the asymptotics from this section are refinements of those from \cite[Theorem 3.4]{GS}. Moreover, we obtain our asymptotics by bootstrapping off the latter. 

\begin{definition}
Define
\[ H_\mu(u) := - (u+1)\log S_\mu(u) - \int \log \left( (u+1) S_\mu(u)^{-1} -u x \right) d\mu(x) \]
where the logarithms are given by the standard branch.
\end{definition}

Observe that
\[ H_\mu(u) = -(u+1)\log(u+1) + u \log u - u \log \psi_\mu^{-1}(u) - \int \log (1 - x \psi_\mu^{-1}(u)) d\mu(x). \]
Using the fact that
\[ -u - \int \frac{x \psi_\mu^{-1}(u)}{1 - x \psi_\mu^{-1}(u)} d\mu(x) = u - \psi_\mu(\psi_\mu^{-1}(u)) = 0, \]
we have
\begin{align} \label{eq:H'}
H_\mu'(u) = -\log(u+1) + \log u - \log \psi_\mu^{-1}(u) = - \log S_\mu(u)
\end{align}
and
\begin{align} \label{eq:H''}
H_\mu''(u) = -\frac{S_\mu'(u)}{S_\mu(u)}.
\end{align}

\begin{definition}
Let $\cR^N$ denote the subset of $\cM_I$ consisting of probability measures of the form
\[ \frac{1}{N} \sum_{i=1}^N \delta_{x_i} \]
where $x_1,\ldots,x_N \in I$.
\end{definition}

\begin{theorem} \label{thm:bessel_asymptotics}
Fix a closed interval $I \subset \R_{>0}$. There exists an open neighborhood $U$ of $[-1,0]$ such that
\begin{align*}
& B_\mu(N(u_1+1),\ldots,N(u_k+1);N(v_1+1),\ldots,N(v_k+1)) \\
& = \frac{\det\left( \frac{1}{\psi_\mu^{-1}(u_i) - \psi_\mu^{-1}(v_j)}\right)_{1 \le i,j \le k}}{\det\left( \frac{1}{u_i - v_j}\right)_{1 \le i,j \le k}} \prod_{i=1}^k \left[ \frac{1}{\sqrt{\psi_\mu'(\psi_\mu^{-1}(u_i)) \psi_\mu'(\psi_\mu^{-1}(v_i))}} \frac{\sqrt{S_\mu(v_i)}e^{N H_\mu(u_i)}}{\sqrt{S_\mu(u_i)} e^{N H_\mu(v_i)}} (1 + o(|u_i - v_i|)) \right]
\end{align*}
as $N\to\infty$, uniformly over $\mu \in \cM_I \cap \cR^N$, $u_1,\ldots,u_k \in U$, and $v_1,\ldots,v_k \in \tfrac{1}{N}\Z \cap [-1,0]$.
\end{theorem}

\begin{remark}
To translate between our notation and that of \cite{GS}, our $H_\mu$ corresponds to their $\wt{\Psi}_{\rho_N}$ and our $\psi_\mu$ corresponds to their $M_{\wt{\rho}_N}$.
\end{remark}

\begin{remark}
We note the peculiarity in \Cref{thm:bessel_asymptotics} that the uniformity $\mu \in \cM_I \cap \cR^N$ is over a set varying with $N$.
\end{remark}

\begin{proof}[Proof of \Cref{thm:bessel_asymptotics}]
Our starting point is \cite[Theorem 3.4]{GS} which states that there is some neighborhood $U$ of $[-1,0]$ such that
\begin{align*}
& B_\mu^{(N)}(N(u_1+1),\ldots,N(u_k+1);N(v_1+1),\ldots,N(v_k+1)) \\
& = \frac{\det\left( \frac{1}{\psi_\mu^{-1}(u_i) - \psi_\mu^{-1}(v_j)}\right)_{1 \le i,j \le k}}{\det\left( \frac{1}{u_i - v_j}\right)_{1 \le i,j \le k}} \prod_{i=1}^k \left[ \frac{1}{\sqrt{\psi_\mu'(\psi_\mu^{-1}(u_i)) \psi_\mu'(\psi_\mu^{-1}(v_i))}} \frac{\sqrt{S_\mu(v_i)}e^{N H_\mu(u_i)}}{\sqrt{S_\mu(u_i)} e^{N H_\mu(v_i)}} (1 + o(1)) \right]
\end{align*}
as $N\to\infty$, uniformly for $u_1,\ldots,u_k,v_1,\ldots,v_k \in U$ and $\mu \in \cM_I \cap \cR^N$. We note that the original statement of \cite[Theorem 3.4]{GS} is in the regime where $\mu = \mu_N$ converges weakly to a measure in $\cM_I$ as $N\to\infty$, but the proof also implies uniform asymptotics for $\mu \in \cM_I \cap \cR^N$. Thus, it remains to improve the relative $o(1)$ error.

Define
\begin{align*}
\begin{split}
& \fB_\mu^{(N)}(u_1,\ldots,u_k;v_1,\ldots,v_k) \\
& \quad \quad = \frac{\det\left( \frac{1}{\psi_\mu^{-1}(u_i) - \psi_\mu^{-1}(v_j)} \right)_{1 \le i,j \le k}}{\det\left( \frac{1}{u_i - v_j}\right)_{1 \le i,j \le k}} \prod_{i=1}^k \left[ \frac{1}{\sqrt{\psi_\mu'(\psi_\mu^{-1}(u_i)) \psi_\mu'(\psi_\mu^{-1}(v_i))}} \frac{\sqrt{S_\mu(v_i)}e^{N H_\mu(u_i)}}{\sqrt{S_\mu(u_i)} e^{N H_\mu(v_i)}} \right]
\end{split}
\end{align*}
for $u_1,\ldots,u_k,v_1,\ldots,v_k \in U$. Then for $\e > 0$ sufficiently small,
\[ B_\mu^{(N)}(N(u_1+1),\ldots,N(u_k+1);N(v_1+1),\ldots,N(v_k+1)) = \fB_\mu^{(N)}(u_1,\ldots,u_k;v_1,\ldots,v_k) (1 + o(1)) \]
as $N\to\infty$, uniformly over $u_1,\ldots,u_k \in U_\e$, and $v_1,\ldots,v_k \in \tfrac{1}{N}\Z \cap [-1,0]$, where $U_\e$ denotes the $\e$-neighborhood of $[-1,0]$.

By \Cref{thm:psi_cauchy}, the quotient of Cauchy determinants in the definition of $\fB_\mu^{(N)}$ is bounded and bounded away from $0$ for $u_1,\ldots,u_k,v_1,\ldots,v_k \in U_\e$, for $\e > 0$ sufficiently small. Similarly, since (see \Cref{thm:psi})
\[ \psi_\mu'(\psi_\mu^{-1}(0)) = \psi_\mu'(0)\ne 0 \quad \quad \mbox{and} \quad \quad S_\mu(0) \ne 0, \]
we have
\[ \frac{1}{\sqrt{\psi_\mu'(\psi_\mu^{-1}(u))\psi_\mu'(\psi_\mu^{-1}(v))}} \frac{\sqrt{S_\mu(v)}}{\sqrt{S_\mu(u)}} \]
is bounded and bounded away from $0$ for $u_1,\ldots,u_k,v_1,\ldots,v_k \in U_\e$, given that $\e$ is sufficiently small. For each integer $k \ge 1$, define
\begin{align*}
F_k^{(N)}(u_1,\ldots,u_k;v_1,\ldots,v_k) &:= \log\left( \frac{B_\mu^{(N)}(N(u_1+1),\ldots,N(u_k+1);N(v_1+1),\ldots,N(v_k+1))}{B_\mu^{(N)}(N(u_1+1),\ldots,N(u_{k-1}+1);N(v_1+1),\ldots,N(v_{k-1}+1))} \right) \\
\fF_k^{(N)}(u_1,\ldots,u_k;v_1,\ldots,v_k) &:= \log\left( \frac{\fB_\mu^{(N)}(Nu_1,\ldots,Nu_k;Nv_1,\ldots,Nv_k)}{\fB_\mu^{(N)}(Nu_1,\ldots,Nu_{k-1};Nv_1,\ldots,Nv_{k-1})} \right)
\end{align*}
where in the case $k = 1$, we take the denominator in the logarithm to be $1$. Then $F_k^{(N)}$ and $\fF_k^{(N)}$ are analytic for $u_1,\ldots,u_k \in U_\e$, where $v_1,\ldots,v_k \in \tfrac{1}{N}\Z \cap [-1,0]$ and $N$ is sufficiently large. Furthermore, $F_k^{(N)}$ and $\fF_k^{(N)}$ vanish whenever $u_k = v_k$.

Then
\[ \frac{1}{u_i - v_i} F_i^{(N)}(u_1,\ldots,u_i;v_1,\ldots,v_i) \quad \quad \mbox{and} \quad \quad \frac{1}{u_i - v_i} \fF_i^{(N)}(u_1,\ldots,u_i;v_1,\ldots,v_i) \]
are analytic for $u_1,\ldots,u_i \in U_\e$, where $v_1,\ldots,v_i \in \tfrac{1}{N} \Z \cap [-1,0]$ and $N$ is sufficiently large. Restricting to $|u_1| = \cdots = |u_i| = 2\e/3$, we have
\[ \frac{1}{u_i - v_i} F_i^{(N)}(u_1,\ldots,u_i;v_1,\ldots,v_i) - \frac{1}{u_i - v_i} \fF_i^{(N)}(u_1,\ldots,u_i;v_1,\ldots,v_i) = o(1). \]
uniformly for $u_1,\ldots,u_i \in \partial U_{2\e/3}$ and $v_1,\ldots,v_i \in \tfrac{1}{N} \Z \cap [-1,0]$. By Cauchy integral formula, the estimate above is valid for $u_1,\ldots,u_i \in U_{\e/2}$ and $v_1,\ldots,v_i \in \tfrac{1}{N} \Z \cap [-1,0]$. Therefore,
\begin{align*}
& B_\mu^{(N)}(N(u_1+1),\ldots,N(u_k+1);N(v_1+1),\ldots,N(v_k+1)) \\
&\quad = \exp\left( \sum_{i=1}^k F_i^{(N)}(u_1,\ldots,u_i;v_1,\ldots,v_i) \right) \\
&\quad = \exp\left( \sum_{i=1}^k \Big( \fF_i^{(N)}(u_1,\ldots,u_i;v_1,\ldots,v_i) + o(|u_i - v_i|) \Big) \right) \\
&\quad = \fB_\mu^{(N)}(u_1,\ldots,u_k;v_1,\ldots,v_k) (1 + o(\max_i |u_i - v_i|))
\end{align*}
as $N\to\infty$, uniformly for $u_1,\ldots,u_k \in U_{\e/2}$ and $v_1,\ldots,v_k \in \tfrac{1}{N} \Z \cap [-1,0]$. This completes the proof of \Cref{thm:bessel_asymptotics}.
\end{proof}

\section{Proof of Theorem \ref{thm:main}} \label{sec:main_proof}

In this section, we prove our main result \Cref{thm:main}. Throughout this section, we fix some notation. Given a sequence $X^{(N)}(1),X^{(N)}(2),\ldots$, denote by $\mu_N^{(m)}$ the empirical distribution of the squared singular values of $X^{(N)}(m)$. Given a compactly supported probability measure $\mu$, let $\kappa_1(\mu)$ and $\kappa_2(\mu)$ denote the mean (first cumulant) and variance (second cumulant) of $\mu$ respectively.

The key step is to establish the following intermediate result.

\begin{theorem} \label{thm:convergence_of_moments}
Suppose that the hypotheses of \Cref{thm:main} (i.e. conditions \eqref{eq:interval_containment} and \eqref{eq:variance_curve}) hold such that $X^{(N)}(1),X^{(N)}(2),\ldots$ have deterministic squared singular values, all contained in a fixed compact interval $I \subset \R_{>0}$. Let
\[ y_1^{(N)}(M) \ge \cdots \ge y_N^{(N)}(M) \]
denote the squared singular values of $X^{(N)}(M) \cdots X^{(N)}(1)$. Then for any $t_1 \ge \cdots \ge t_k > 0$ and $c_1,\ldots,c_k > 0$ such that $c_1 + \cdots + c_k \in (0,1)$, we have
\[ \E\left[ \prod_{i=1}^k \sum_{j=1}^N e^{c_i \log y_j^{(N)}(\lfloor t_i N \rfloor)} \right] = \left( \prod_{i=1}^k e^{c_i \cE_N(\lfloor t_i N \rfloor)} \right) \E\left[ \prod_{i=1}^k \sum_{j=1}^N e^{c_i \left( \xi_j^{(N)}(\frac{1}{4}\cV_N(\lfloor t_i N \rfloor)) - \frac{N}{2} \cV_N(\lfloor t_i N \rfloor)\right)} \right] (1 + o(1)) \]
as $N\to\infty$, where
\[ \cE_N(M):= \sum_{m=1}^M \log \kappa_1(\mu_N^{(m)}), \quad \quad \mbox{and} \quad \quad \cV_N(M) := \frac{1}{N} \sum_{m=1}^M \frac{\kappa_2(\mu_N^{(m)})}{\kappa_1(\mu_N^{(m)})^2}. \]
This convergence holds uniformly over sequences $X_1^{(N)},X_2^{(N)},\ldots$ satisfying \eqref{eq:interval_containment} and \eqref{eq:variance_curve} such that the squared singular values of $X^{(N)}(m)$ lie in $I$ for every $1 \le i \le M_1$.
\end{theorem}

The proof of \Cref{thm:convergence_of_moments} combines the asymptotics from the previous sections and our formalism of multivariate Bessel functions. Note that \Cref{thm:convergence_of_moments} makes the assumption that the matrices have \emph{non-random} singular values. The proof of \Cref{thm:main} proceeds straightforwardly from \Cref{thm:convergence_of_moments} by bootstrapping from the deterministic case, see \Cref{ssec:proof_of_main}.

\subsection{Proof of Theorem \ref{thm:convergence_of_moments}}

Let $M_i := M_i(N) := \lfloor t_i N \rfloor$ for $1 \le i \le k$ and $M_{k+1} := 0$. Let $\vec{x}(m) = (x_1^{(m)},\ldots,x_N^{(m)})$ denote the squared singular values of $X^{(N)}(m)$. By \Cref{thm:observable_operators},
\begin{align*}
& \E\left[ \prod_{i=1}^k \sum_{j=1}^N e^{c_i\log y_j^{(N)}(M_i)} \right] \\
& \quad = \left. \cD_{c_1}^{(N)} \prod_{m_1=M_2+1}^{M_1} \frac{\cB_{\log \vec{x}(m_1)}(z_1,\ldots,z_N)}{\cB_{\log \vec{x}(m_1)}(\rho_N)} \cdots \cD_{c_k}^{(N)} \prod_{m_k=M_{k+1}+1}^{M_k} \frac{\cB_{\log \vec{x}(m_k)}(z_1,\ldots,z_N)}{\cB_{\log \vec{x}(m_k)}(\rho_N)} \right|_{\vec{z} = \rho_N}.
\end{align*}
Recall our convention that $\cD_c = \cD_c^{(N)}$ acts on everything to its right (see \Cref{sec:operators}). Expanding out the $\cD_c$ terms, we obtain
\begin{align*}
\E\left[ \prod_{i=1}^k \sum_{j=1}^N e^{c_i\log y_j^{(N)}(M_i)} \right] = & \sum_{i_1,\ldots,i_k=1}^N \left( \prod_{j_1 \ne i_1} \frac{c_1 + z_{i_1} - z_{j_1}}{z_{i_1} - z_{j_1}} \right) \cT_{c_1,z_{i_1}}  \prod_{m_1 = M_2+1}^{M_1} \frac{\cB_{\log \vec{x}(m_1)}(z_1,\ldots,z_N)}{\cB_{\log \vec{x}(m_1)}(\rho_N)} \\
& \left. \cdots \left( \prod_{j_k \ne i_k} \frac{c_k + z_{i_k} - z_{j_k}}{z_{i_k} - z_{j_k}} \right) \cT_{c_k,z_{i_k}} \prod_{m_k = M_{k+1} + 1}^{M_k} \frac{\cB_{\log \vec{x}(m_k)}(z_1,\ldots,z_N)}{\cB_{\log \vec{x}(m_k)}(\rho_N)} \right|_{\vec{z} = \rho_N}.
\end{align*}
The products over $j_\ell \ne i_\ell$ are understood to range over $1 \le j_\ell \le N$ (for $1 \le \ell \le k$). Like $\cD_c$, the shift operators $\cT_{c,z_i}$ act on everything to the right of it. If the $\cT_{c,z_i}$ is contained between parentheses, its action is confined within those parentheses.

Since $\cT_c fg  = (\cT_c f)(\cT_c g)$, we get
\begin{align*}
\E\left[ \prod_{i=1}^k \sum_{j=1}^N e^{c_i\log y_j^{(N)}(M_i)} \right] = \sum_{i_1,\ldots,i_k=1}^N \sigma_{i_1,\ldots,i_k}
\end{align*}
where
\begin{align*}
\sigma_{i_1,\ldots,i_k} =& \prod_{\ell=1}^k \left( \left. \cT_{c_1,z_{i_1}} \cdots \cT_{c_{\ell-1},z_{i_{\ell-1}}} \prod_{j_\ell \ne i_\ell} \frac{c_\ell + z_{i_\ell} - z_{j_\ell}}{z_{i_\ell} - z_{j_\ell}} \right|_{\vec{z} = \rho_N} \right) \\
& \quad \times \left( \prod_{m_\ell=M_{\ell+1}+1}^{M_\ell} \left. \cT_{c_1,z_{i_1}} \cdots \cT_{c_\ell,z_{i_\ell}} \frac{\cB_{\log \vec{x}(m_k)}(z_1,\ldots,z_N)}{\cB_{\log \vec{x}(m_k)}(\rho_N)} \right|_{\vec{z} = \rho_N} \right).
\end{align*}

Set
\begin{align*}
& \tau_{i_1,\ldots,i_k} := \left( \prod_{i=1}^k e^{c_i\cE_N(M_i)} \right) \sum_{i_1,\ldots,i_k=1}^N \left( \prod_{j_1 \ne i_1} \frac{c_1 + z_{i_1} - z_{j_1}}{z_{i_1} - z_{j_1}} \right) \cT_{c_1,z_{i_1}} \left( \prod_{a_1=1}^N  \frac{\exp\left[ \Delta_1 \left(z_{a_1} - N + \frac{1}{2}\right)^2\right]}{\exp\left[ \Delta_1 \left(-a_1 + \frac{1}{2}\right)^2\right]} \right) \\
& \quad \quad \quad \quad \quad \quad \times \cdots \times \left. \left( \prod_{j_k \ne i_k} \frac{c_k + z_{i_k} - z_{j_k}}{z_{i_k} - z_{j_k}} \right) \cT_{c_k,z_{i_k}} \left( \prod_{a_k=1}^N \frac{\exp\left[ \Delta_k \left(z_{a_k} - N + \frac{1}{2}\right)^2\right]}{\exp\left[ \Delta_k \left(-a_k + \frac{1}{2}\right)^2\right]} \right) \right|_{\vec{z} = \rho_N} \\
&\quad = \left( \prod_{i=1}^k e^{c_i \cE_N(M_i)} \right) \cD_{c_1} \left( \prod_{a_1=1}^N \frac{\exp\left[ \Delta_1 \left( z_{a_1} - N + \frac{1}{2}\right)^2 \right]}{\exp\left[ \Delta_1 \left( -a_1 + \frac{1}{2}\right)^2 \right]} \right)\cdots \left. \cD_{c_k} \left( \prod_{a_k=1}^N \frac{\exp\left[ \Delta_k \left( z_{a_k} - N + \frac{1}{2}\right)^2 \right]}{\exp\left[ \Delta_k \left( -a_k + \frac{1}{2}\right)^2 \right]} \right) \right|_{\vec{z} = \rho_N}
\end{align*}
where
\[ \Delta_\ell := \frac{1}{2N} \sum_{m=M_{\ell+1}+1}^{M_\ell} \frac{\kappa_2(\mu_N^{(m)})}{\kappa_1(\mu_N^{(m)})^2} = \frac{1}{2}\left( \cV_N(M_\ell) - \cV_N(M_{\ell+1}) \right), \quad \quad 1 \le \ell \le k. \]
The equality following the definition of $\tau_{i_1,\ldots,i_k}$ follows from the definition for $\cD_c$, as in the calculation (though in reverse) at the start of this proof.

We prove that
\begin{align} \label{eq:sigma_tau_replace}
\sigma_{i_1,\ldots,i_k} = \tau_{i_1,\ldots,i_k} (1 + o(1))
\end{align}
as $N\to\infty$, uniformly over $1 \le i_1,\ldots,i_k \le N^{1/3}$. Furthermore, we prove that if $N$ is sufficiently large then
\begin{gather} \label{eq:sigma_positive}
\sigma_{i_1,\ldots,i_k} > 0, \\ \label{eq:tau_positive}
\tau_{i_1,\ldots,i_k} > 0,
\end{gather}
for $1 \le i_1,\ldots,i_k \le N$, and there exists $c > 0$ such that
\begin{gather} \label{eq:sigma_comparison}
\sigma_{i_1,\ldots,i_k} \le \sigma_{1,\ldots,1} e^{-c N^{1/3}}, \\ \label{eq:tau_comparison}
\tau_{i_1,\ldots,i_k} \le \sigma_{1,\ldots,1} e^{-c N^{1/3}}
\end{gather}
for $1 \le i_1,\ldots,i_k \le N$ such that $i_j > N^{1/3}$ for some $1 \le j \le k$. Indeed, \Cref{thm:convergence_of_moments} would follow because
\begin{align*}
\E\left[ \prod_{i=1}^k \sum_{j=1}^N e^{c_i \log y_j^{(N)}(M_i)} \right] =& \sum_{i_1,\ldots,i_k=1}^N \sigma_{i_1,\ldots,i_k} \\
=& (1 + o(1)) \sum_{i_1,\ldots,i_k=1}^N \tau_{i_1,\ldots,i_k} \\
=& (1 + o(1)) \left( \prod_{i=1}^k e^{c_i\cE_N(M_i)} \right) \cD_{c_1} \left( \prod_{a_1=1}^N  \frac{\exp\left[ \Delta_1 \left(z_{a_1} - N + \frac{1}{2}\right)^2\right]}{\exp\left[ \Delta_1 \left(-a_1 + \frac{1}{2}\right)^2\right]} \right) \\
& \quad \quad \left.\cdots \cD_{c_k} \left( \prod_{a_k=1}^N \frac{\exp\left[ \Delta_k \left(z_{a_k} - N + \frac{1}{2}\right)^2\right]}{\exp\left[ \Delta_k \left(-a_k + \frac{1}{2}\right)^2\right]} \right) \right|_{\vec{z} = \rho_N} \\
=& (1 + o(1)) \left( \prod _{i=1}^k e^{c_i\cE_N(M_i)} \right) \E \left[ \prod_{i=1}^k \sum_{j=1}^N e^{c_i \left( \xi_j^{(N)}\left(\frac{1}{4}\cV_N(M_i)\right) - \frac{N}{2} \cV_N(M_i) \right)} \right].
\end{align*}
The second equality comes from \eqref{eq:sigma_tau_replace} applied to the terms with $i_1,\ldots,i_k \le N^{1/3}$, and the remaining terms are tail terms which can be replaced by \eqref{eq:sigma_positive}-\eqref{eq:tau_comparison}. The third equality follows from the definition of $\tau_{i_1,\ldots,i_k}$ and the definition of $\cD_c$. The fourth equality uses \Cref{thm:brownian_observable}.

Therefore, our goal is to prove \eqref{eq:sigma_tau_replace}-\eqref{eq:tau_comparison}. For this, we rely on the following claims:

\begin{claim} \label{claim:rational_term}
For any $1 \le i_1,\ldots,i_\ell \le N$, we have
\begin{gather} \label{eq:positivity_rational_term}
\left. \cT_{c_1,z_{i_1}} \cdots \cT_{c_{\ell-1},z_{i_{\ell-1}}} \prod_{j \ne i_\ell} \frac{c_\ell + z_{i_\ell} - z_j}{z_{i_\ell} - z_j} \right|_{\vec{z} = \rho_N} > 0 \\ \label{eq:comparison_rational_term}
\left. \cT_{c_1,z_{i_1}} \cdots \cT_{c_{\ell-1},z_{i_{\ell-1}}} \prod_{j \ne i_\ell} \frac{c_\ell + z_{i_\ell} - z_j}{z_{i_\ell} - z_j} \right|_{\vec{z} = \rho_N} \le C \left. \cT_{c_1,z_1} \cdots \cT_{c_{\ell-1},z_1} \prod_{j \ne 1} \frac{c_\ell + z_1 - z_j}{z_1 - z_j} \right|_{\vec{z} = \rho_N}
\end{gather}
for some constant $C > 1$ uniform in the $i_1,\ldots,i_k$ but depending on $c_1,\ldots,c_k > 0$ satisfying $c_1 + \cdots + c_k < 1$.
\end{claim}

\begin{claim} \label{claim:bessel}
Let $\mu_N = \tfrac{1}{N} \sum_{i=1}^N \delta_{x_i}$ and $\vec{x} = (x_1,\ldots,x_N)$. Then
\begin{align} \label{eq:bessel_shift_asymptotics}
\left. \cT_{c_1,z_{i_1}} \cdots \cT_{c_\ell,z_{i_\ell}} \frac{\cB_{\log \vec{x}}(z_1,\ldots,z_N)}{\cB_{\log \vec{x}}(\rho_N)} \right|_{\vec{z} = \rho_N} = \prod_{j=1}^\ell \exp\left( -c_j \log S_{\mu_N}(-\tfrac{i_j}{N}) - \frac{c_j(c_j + 1)}{2N} \frac{S_{\mu_N}'(-\frac{i_j}{N})}{S_{\mu_N}(-\frac{i_j}{N})} \right) (1 + o(N^{-1}))
\end{align}
as $N\to\infty$, uniformly over $\mu_N \in \cM_I \cap \cR^N$ and $i_1,\ldots,i_\ell \in \tfrac{1}{N} \Z \cap [0,1]$. In particular, if $i_1,\ldots,i_\ell \le N^{1/3}$, then
\begin{align} \label{eq:bessel_shift_asymptotics_near0}
\left. \cT_{c_1,z_{i_1}} \cdots \cT_{c_\ell,z_{i_\ell}} \frac{\cB_{\log \vec{x}}(z_1,\ldots,z_N)}{\cB_{\log \vec{x}}(\rho_N)} \right|_{\vec{z} = \rho_N} &= \prod_{j=1}^\ell \kappa_1(\mu_N)^{c_j} \exp\left[ \frac{1}{N} \frac{\kappa_2(\mu_N)}{\kappa_1(\mu_N)^2} \left( -c_j i_j + \frac{c_j(c_j + 1)}{2} \right) \right] (1 + o(N^{-1})).
\end{align}
\end{claim}

Before providing the proofs of Claims \ref{claim:rational_term} and \ref{claim:bessel}, we explain how \eqref{eq:sigma_tau_replace}-\eqref{eq:tau_comparison} follow from the claims. We can rewrite
\begin{align*}
\tau_{i_1,\ldots,i_k} =& \left( \prod_{i=1}^k e^{c_i\cE_N(M_i)} \right) \prod_{\ell=1}^k \left( \left. \cT_{c_1,z_{i_1}} \cdots \cT_{c_{\ell-1},z_{i_{\ell-1}}} \prod_{j_\ell \ne i_\ell} \frac{c_\ell + z_{i_\ell} - z_{j_\ell}}{z_{i_\ell} - z_{j_\ell}} \right|_{\vec{z} = \rho_N} \right) \\
& \quad \times \prod_{m_\ell = M_{\ell+1}+1}^{M_\ell} \left( \left. \cT_{c_1,z_{i_1}} \cdots \cT_{c_\ell,z_{i_\ell}} \prod_{a=1}^N \frac{\exp\left[ \frac{1}{2N} \frac{\kappa_2(\mu_N^{(m_\ell)})}{\kappa_1(\mu_N^{(m_\ell)})^2} \left(z_a - N + \frac{1}{2}\right)^2\right]}{\exp\left[ \frac{1}{2N} \frac{\kappa_2(\mu_N^{(m_\ell)})}{\kappa_1(\mu_N^{(m_\ell)})^2} \left(-a + \frac{1}{2}\right)^2\right]} \right|_{\vec{z} = \rho_N} \right) \\
=& \prod_{\ell=1}^k \left( \left. \cT_{c_1,z_{i_1}} \cdots \cT_{c_{\ell-1},z_{i_{\ell-1}}} \prod_{j_\ell \ne i_\ell} \frac{c_\ell + z_{i_\ell} - z_{j_\ell}}{z_{i_\ell} - z_{j_\ell}} \right|_{\vec{z} = \rho_N} \right) \\
& \quad \times \left(\prod_{m_\ell = M_{\ell+1}+1}^{M_\ell} \prod_{j=1}^\ell \kappa_1(\mu_N^{(m_\ell)})^{c_j} \exp\left[ \frac{1}{N} \frac{\kappa_2(\mu_N^{(m_\ell)})}{\kappa_1(\mu_N^{(m_\ell)})^2} \left(-c_j i_j + \frac{c_j(c_j+1)}{2} \right) \right] \right)
\end{align*}
Then \eqref{eq:bessel_shift_asymptotics_near0} and the definition of $\sigma_{i_1,\ldots,i_k}$ immediately imply \eqref{eq:sigma_tau_replace}. The positivity statements \eqref{eq:sigma_positive} and \eqref{eq:tau_positive} also follow from \eqref{eq:positivity_rational_term} and \eqref{eq:bessel_shift_asymptotics}.

We must still prove \eqref{eq:sigma_comparison} and \eqref{eq:tau_comparison}. To prove \eqref{eq:tau_comparison}, notice that the last line in the expression above for $\tau_{i_1,\ldots,i_k}$ is strictly decreasing in each of $i_1,\ldots,i_k$, and the decay is exponential. Using the fact that $M_\ell - M_{\ell+1}$ is $O(N)$, and that $M_k - M_{k+1} = M_k$ grows linearly with $N$, we have that if $N$ is sufficiently large, then there exists $c > 0$ such that
\[ \tau_{i_1,\ldots,i_k} \le \tau_{1,\ldots,1} e^{-cN^{1/3}} \]
for $1 \le i_1,\ldots,i_k \le N$ such that $i_j > N^{1/3}$ for some $1 \le j \le k$. By \eqref{eq:sigma_tau_replace} and \Cref{claim:rational_term}, we obtain \eqref{eq:tau_comparison}. 

To prove \eqref{eq:sigma_comparison}, observe that by \Cref{claim:bessel},
\begin{align*}
\sigma_{i_1,\ldots,i_k} =& \prod_{\ell=1}^k \left( \left. \cT_{c_1,z_{i_1}} \cdots \cT_{c_{\ell-1},z_{i_{\ell-1}}} \prod_{j_\ell \ne i_\ell} \frac{c_\ell + z_{i_\ell} - z_{j_\ell}}{z_{i_\ell} - z_{j_\ell}} \right|_{\vec{z} = \rho_N} \right) \\
& \quad \times \left(\prod_{m_\ell = M_{\ell+1}+1}^{M_\ell} \prod_{j=1}^\ell \exp\left[ -c_j \log S_{\mu_N^{(m_\ell)}}(-\tfrac{i_j}{N}) - \frac{c_j(c_j+1)}{2N} \frac{S_{\mu_N^{(m_\ell)}}'(-\frac{i_j}{N})}{S_{\mu_N^{(m_\ell)}}(-\frac{i_j}{N})} \right] \right) (1 + o(1)) \\
=& \prod_{\ell=1}^k \left( \left. \cT_{c_1,z_{i_1}} \cdots \cT_{c_{\ell-1},z_{i_{\ell-1}}} \prod_{j_\ell \ne i_\ell} \frac{c_\ell + z_{i_\ell} - z_{j_\ell}}{z_{i_\ell} - z_{j_\ell}} \right|_{\vec{z} = \rho_N} \right) \\
& \quad \times \left( \prod_{j=1}^\ell \exp\left[ \sum_{m = 1}^{M_\ell} \left( -c_j \log S_{\mu_N^{(m)}}(-\tfrac{i_j}{N}) - \frac{c_j(c_j+1)}{2N} \frac{S_{\mu_N^{(m)}}'(-\frac{i_j}{N})}{S_{\mu_N^{(m)}}(-\frac{i_j}{N})} \right) \right] \right) (1 + o(1))
\end{align*}
as $N\to\infty$, uniformly over $1 \le i_1,\ldots,i_k \le N$. We know that $-\log S_\mu(u)$ is an increasing function on $[-1,0]$ with
\[ \left. -\frac{d}{du} \log S_\mu(u) \right|_{u=0} = -\frac{S_\mu'(0)}{S_\mu(0)} = \frac{\kappa_2(\mu)}{\kappa_1(\mu)^2} \]
where we use \Cref{thm:S_and_S'}. Condition \eqref{eq:variance_curve} then implies
\[ -\sum_{m=1}^{M_\ell} \left. \frac{d}{du} \log S_{\mu_N^{(m)}}(u) \right|_{u=0} = \sum_{m=1}^{M_\ell} \frac{\kappa_2(\mu_N^{(m)})}{\kappa_1(\mu_N^{(m)})^2} = \gamma(t_\ell) + o(1) \]
as $N\to\infty$. Recalling that $\gamma$ is a continuous map from $\R_{>0}$ to $\R_{>0}$, we see that the expession above is positive and uniformly bounded away from $0$ in $N$. By \Cref{thm:S_and_S'} again,
\[ -\left. \frac{d^2}{du^2} \log S_\mu(u) \right|_{u=0} = \frac{S_\mu'(0)^2}{S_\mu(0)^2}-\frac{S_\mu''(0)}{S_\mu(0)} \]
which is bounded for $\mu \in \cM_I$, where we use the continuity of $\mu \mapsto S_\mu$ on $\cM_I$ (see \Cref{thm:psi}) and the compactness of $\cM_I$. Thus
\[ - \sum_{m=1}^{M_\ell} \left. \frac{d^2}{du^2} \log S_{\mu_N^{(m)}}(u) \right|_{u=0} \]
is bounded, uniformly in $N$. These considerations imply
\begin{align*}
-\sum_{m=1}^{M_\ell} \log S_{\mu_N^{(m)}}(-\tfrac{i}{N}) &\le -\sum_{m=1}^{M_\ell} \log S_{\mu_N^{(m)}}(-\tfrac{1}{N}) \quad \quad \mbox{for $1 \le i \le N$}, \\
-\sum_{m=1}^{M_\ell} \log S_{\mu_N^{(m)}}(-\tfrac{i}{N}) &\le -\sum_{m=1}^{M_\ell} \log S_{\mu_N^{(m)}}(-\tfrac{1}{N}) - CN^{1/3} \quad \quad \mbox{for $N^{1/3} < i \le N$}
\end{align*} 
for some constant $C > 0$ (since $M_\ell \asymp N$). The term
\[ - \frac{1}{2N} \sum_{m=1}^{M_\ell} \frac{S_{\mu_N^{(m)}}'(-\frac{i}{N})}{S_{\mu_N^{(m)}}(-\frac{i}{N})} \]
is also bounded, uniformly over $1 \le i \le N$ and in $N$, appealing again to \Cref{thm:psi}. Therefore
\begin{align*}
& \sum_{m=1}^{M_\ell} \left( -c_j \log S_{\mu_N^{(m)}}(-\tfrac{i}{N}) - \frac{c_j(c_j+1)}{2N} \frac{S_{\mu_N^{(m)}}'(-\tfrac{i}{N})}{S_{\mu_N^{(m)}}(-\frac{i}{N})} \right) \\
& \quad \quad \le \sum_{m=1}^{M_\ell} \left( -c_j \log S_{\mu_N^{(m)}}(-\tfrac{1}{N}) - \frac{c_j(c_j+1)}{2N} \frac{S_{\mu_N^{(m)}}'(-\frac{1}{N})}{S_{\mu_N^{(m)}}(-\frac{1}{N})} \right) - C' N^{1/3}
\end{align*}
for $1 \le i \le N$, for some $C' > 0$. Combining this with \Cref{claim:rational_term} implies \eqref{eq:sigma_comparison}.

Having justified that Claims \ref{claim:rational_term} and \ref{claim:bessel} imply \eqref{eq:sigma_tau_replace}-\eqref{eq:tau_comparison}, it remains to prove these claims.

\begin{proof}[Proof of \Cref{claim:rational_term}]
Let $i_1',\ldots,i_{r-1}'$ be the distinct elements of $\{i_1,\ldots,i_{\ell-1}\} \setminus \{i_\ell\}$ and set $i_r' = i_\ell$. Define
\[ c_1' := \sum_{\substack{1 \le a < r \\ i_a = i_1'}} c_a, \quad \ldots, \quad c_r' := \sum_{\substack{1 \le a < r \\ i_a = i_r'}} c_a, \]
where we note that $c_1',\ldots,c_{r-1}' > 0$, but $c_r' > 0$ if and only if $i_r' = i_\ell$ is among $i_1,\ldots,i_{\ell-1}$, i.e. $i_\ell \in \{i_1,\ldots,i_{\ell-1}\}$. Then
\begin{align*}
& \cT_{c_1,z_{i_1}} \cdots \cT_{c_{\ell-1},z_{i_{\ell-1}}} \prod_{j \ne i_\ell} \frac{c_\ell + z_{i_\ell} - z_j}{z_{i_\ell} - z_j} = \cT_{c_1',z_{i_1'}} \cdots \cT_{c_r',z_{i_r'}} \prod_{j \ne i_r'} \frac{c_\ell + z_{i_r'} - z_j}{z_{i_r'} - z_j} \\
&\quad \quad = \left( \prod_{a = 1}^{r-1} \frac{(z_{i_r'} - z_{i_a'})(c_\ell + c_r' + z_{i_r'} - c_a' - z_{i_a'})}{(c_\ell + c_r' + z_{i_r'} - z_{i_a'})(z_{i_r'} - c_a' - z_{i_a'})} \right)\left( \prod_{j \ne i_r'} \frac{c_\ell + c_r' + z_{i_r'} - z_j}{z_{i_r'} - z_j} \right).
\end{align*}
Since $0 < c_1 + \cdots + c_\ell < 1$, upon evaluating at $\vec{z} = \rho_N = (N-1,N-2,\ldots,0)$, we see that the expression above is positive. This proves \eqref{eq:positivity_rational_term}. Furthermore, we have a bound
\[ \left. \prod_{a = 1}^{r-1} \frac{(z_{i_r'} - z_{i_a'})(c_\ell + c_r' + z_{i_r'} - c_a' - z_{i_a'})}{(c_\ell + c_r' + z_{i_r'} - z_{i_a'})(z_{i_r'} - c_a' - z_{i_a'})} \right|_{\vec{z} = \rho_N} \le C \]
where we may make $C$ uniform over $1 \le i_1,\ldots,i_\ell \le N$ by virtue of $0 < c_1 + \cdots + c_\ell < 1$ and the positivity of $c_1,\ldots,c_\ell$. Thus
\begin{align*}
\cT_{c_1,z_{i_1}} \cdots \cT_{c_{\ell-1},z_{i_{\ell-1}}} \prod_{j \ne i_\ell} \frac{c_\ell + z_{i_\ell} - z_j}{z_{i_\ell} - z_j} \le C \left. \prod_{j \ne i_r'} \frac{c_\ell + c_r' + z_{i_r'} - z_j}{z_{i_r'} - z_j} \right|_{\vec{z} = \rho_N}.
\end{align*}
Next, observe that
\begin{align*}
\left. \prod_{j \ne i} \frac{c_\ell + c_r' + z_i - z_j}{z_i - z_j} \right|_{\vec{z} = \rho_N} &= \prod_{j \ne i} \frac{c_\ell + c_r' + j - i}{j - i} \\
&\le \prod_{j \ne 1} \frac{c_\ell + c_r' + j - 1}{j - 1} \\
& \le \prod_{j \ne 1} \frac{c_1 + \cdots + c_\ell + j - 1}{j - 1} \\
& = \left. \cT_{c_1,z_1} \cdots \cT_{c_{\ell-1},z_1} \prod_{j \ne 1} \frac{c_\ell + z_1 - z_j}{z_1 - z_j} \right|_{\vec{z} = \rho_N}
\end{align*}
where the third line uses the fact that $c_\ell + c_r' \le c_1 + \cdots + c_\ell$, recalling that products are restricted over $1 \le j \le N$. Combining these inequalities, the claim follows.
\end{proof}

\begin{proof}[Proof of \Cref{claim:bessel}]
\Cref{thm:bessel_asymptotics} and \Cref{thm:psi_cauchy} imply
\begin{align*}
\left. \cT_{c_1,z_{i_1}} \cdots \cT_{c_\ell,z_{i_\ell}} \frac{\cB_{\log \vec{x}}(z_1,\ldots,z_N)}{\cB_{\log \vec{x}}(\rho_N)} \right|_{\vec{z} = \rho_N} &= B_{\mu_N}(N - i_1 + c_1,\ldots,N - i_\ell + c_\ell; N - i_1,\ldots, N - i_\ell) \\
&= \prod_{j=1}^\ell \frac{\sqrt{S_{\mu_N}(\frac{-i_j}{N})} e^{N H_{\mu_N}(\frac{-i_j + c_j}{N})} }{\sqrt{S_{\mu_N}(\frac{-i_j + c_j}{N})} e^{N H_{\mu_N}(\frac{-i_j}{N})}} (1 + o(N^{-1})) \\
&= \prod_{j=1}^\ell \exp\left( -c_j \log S_{\mu_N}(-\tfrac{i_j}{N}) - \frac{c_j(c_j + 1)}{2N} \frac{S_{\mu_N}'(-\frac{i_j}{N})}{S_{\mu_N}(-\frac{i_j}{N})} \right) (1 + o(N^{-1}))
\end{align*}
uniformly over $\mu_N \in \cM_I \cap \cR^N$ and $i_1,\ldots,i_\ell \in \tfrac{1}{N} \Z \cap [0,1]$. Note that the final equality follows from the estimates
\begin{align*}
N\Big( H_\mu(\tfrac{u + c}{N}) - H_\mu(\tfrac{u}{N}) \Big) &= c H_\mu'(\tfrac{u}{N}) + \frac{c^2}{2N} H_\mu''(\tfrac{u}{N}) + O(N^{-2}) \\
\frac{\sqrt{S_{\mu_N}(\frac{u}{N})}}{\sqrt{S_{\mu_N}(\frac{u+c}{N})}} &= \exp\left( \frac{1}{2} \left( \log S_{\mu_N}(\tfrac{u}{N}) - \log S_{\mu_N}(\tfrac{u+c}{N}) \right) \right) \\
&= \exp\left( -\frac{c}{2N} \frac{S_{\mu_N}'(\tfrac{u}{N})}{S_{\mu_N}(\tfrac{u}{N})} + O(N^{-2}) \right)
\end{align*}
which hold for fixed $c > 0$, uniformly over $u \in [-1,0]$ and $\mu \in \cM_I$ as $N\to\infty$, and from expressing
$H_\mu'$ and $H_\mu''$ in terms of the $S$-transform as in \eqref{eq:H'} and \eqref{eq:H''}. This proves \eqref{eq:bessel_shift_asymptotics}.

If $|u| \le N^{1/3}$, then
\begin{align*}
-c\log S_\mu(u) - \frac{c(c+1)}{2N} \frac{S_\mu'(\frac{u}{N})}{S_\mu(\frac{u}{N})} &= -c \log S_\mu(0) - \frac{1}{N} \frac{S_\mu'(0)}{S_\mu(0)} \left( u + \frac{c(c+1)}{2} \right) + O(N^{-4/3}) \\
&= c \log \kappa_1(\mu) + \frac{1}{N} \frac{\kappa_2(\mu)}{\kappa_1(\mu)^2} \left( u + \frac{c(c+1)}{2} \right) + O(N^{-4/3})
\end{align*}
uniformly in $u$ and $\mu \in \cM_I$, where the second equality uses the evaluations of $S_\mu(0)$ and $S_\mu'(0)$ from \Cref{thm:S_and_S'}. This proves \eqref{eq:bessel_shift_asymptotics_near0}.
\end{proof}

\subsection{Proof of Theorem \ref{thm:main}} \label{ssec:proof_of_main}
Let
\[ \cE_N(M):= \sum_{m=1}^M \log \kappa_1(\mu_N^{(m)}) \quad \quad \mbox{and} \quad \quad \cV_N(M) := \frac{1}{N} \sum_{m=1}^M \frac{\kappa_2(\mu_N^{(m)})}{\kappa_1(\mu_N^{(m)})^2} \]
as in \Cref{thm:convergence_of_moments}. In contrast with the setting of \Cref{thm:convergence_of_moments}, $\cE_N(M)$ and $\cV_N(M)$ are not deterministic in general.

Fix $t_1 \ge \cdots \ge t_k > 0$. Our goal is to show that for any positive integers $k$ and $h$,
\begin{align} \label{eq:finite_dimensional_distribution}
\begin{split}
& \lim_{N\to\infty} \PP\left(\log y_j^{(N)}(\lfloor t_i N \rfloor) - \cE_N (\lfloor t_i N \rfloor) - \log N \le a_{i,j}: 1 \le i \le k, 1 \le j \le h\right) \\
& \quad \quad \quad \quad \quad \quad \quad \quad = \PP(\xi_j(\gamma(t_i)) \le a_{i,j}: 1 \le i \le k, 1 \le j \le h)
\end{split}
\end{align}
for every $a_{i,j} \in \R$ such that $a_{i,1} \ge \cdots \ge a_{i,h}$.

We may assume that there exists a compact interval $I \subset \R_{>0}$ such that
\begin{align} \label{eq:reduction}
\supp \mu_N(m) \subset I, \quad \quad 1 \le m \le \lfloor t_1N \rfloor, \quad \quad N \ge 1.
\end{align}
We show why this reduction is valid. Indeed, by \eqref{eq:interval_containment}
\[ 1 - \PP\left( \supp \mu_N(m) \subset I \right) = o(1/N) \]
uniformly over $m = 1,2,\ldots$. Then if $\cI_N$ is the event that $\supp \mu_N(m) \subset I$ for every $1 \le m \le \lfloor t_1 N \rfloor$, we have
\[ 1 - \PP\left( \cI_N \right) \le \lfloor t_1 N \rfloor \cdot o(1/N) = o(1). \]
Thus, to prove \eqref{eq:finite_dimensional_distribution}, we may assume without loss of generality that $\PP(\cI_N) = 1$ and that the complement of $\cI_N$ is empty, which is the desired reduction.

Let
\[ \cX_N := \left\{ x_j^{(N)}(m): 1 \le m \le M_1, 1 \le j \le N \right\} \]
denote the collection of squared singular values of $X^{(N)}(1),\ldots,X^{(N)}(M_1)$. We have $\cX_N \subset I$ from \eqref{eq:reduction}.

Under the assumption $\PP(\cI_N) = 1$, we have
\begin{align} \label{eq:EV_uniform}
\cV_N(\lfloor t_i N \rfloor) = \gamma(t_i) + o(1)
\end{align}
as $N\to\infty$ for $1 \le i \le k$. The point is that the $o_{\PP}(1)$ is upgraded to $o(1)$ uniform over all realization of $\cX_N$.

Suppose $c_1,\ldots,c_k > 0$ such that $c_1 + \cdots + c_k < 1$. Then
\begin{align*}
& \E\left[ \left. \prod_{i=1}^k \sum_{j=1}^N e^{c_i \left( \log y_j^{(N)}(\lfloor t_i N \rfloor) - \cE_N(\lfloor t_i N \rfloor) - \log N \right)} \right| \cX_N \right] \\
&\quad \quad = \E \left[ \left. \prod_{i=1}^k \sum_{j=1}^N e^{c_i\left( \xi_j^{(N)}(\frac{1}{4}\cV_N(\lfloor t_i N \rfloor)) - \frac{N}{2} \cV_N(\lfloor t_i N \rfloor) - \log N \right)} \right| \cX_N \right] (1 + o(1)) \\
&\quad \quad = \E \left[ \left. \prod_{i=1}^k \sum_{j=1}^\infty e^{c_i \xi_j(\cV_N(\lfloor t_i N \rfloor))} \right| \cX_N \right] (1 + o(1)) \\
&\quad \quad = \E \left[ \prod_{i=1}^k \sum_{j=1}^\infty e^{c_i \xi_j(\gamma(t_i))} \right] (1 + o(1))
\end{align*}
as $N\to\infty$, uniformly over all realizations of $\cX_N$. The first equality follows from \Cref{thm:convergence_of_moments}, using the uniformity statement in that theorem along with fact that $\cX_N \subset I$. The second equality uses \Cref{thm:brownian_convergence} and \Cref{thm:existence}. The third equality follows from \eqref{eq:EV_uniform}. Taking an overall expectation then yields
\begin{align} \label{eq:laplace_converge}
\lim_{N\to\infty} \E\left[ \prod_{i=1}^k \sum_{j=1}^N e^{c_i \left( \log y_j^{(N)}(\lfloor t_i N \rfloor) - \cE_N(\lfloor t_i N \rfloor) - \log N \right)} \right] = \E \left[ \prod_{i=1}^k \sum_{j=1}^\infty e^{c_i \xi_j(\gamma(t_i))} \right].
\end{align}
The result now follows from \Cref{thm:laplace_implies_findim}

\bibliographystyle{alpha_abbrvsort}
\bibliography{mybib}

\end{document}